\documentclass[10pt,a4paper,reqno]{amsart} 

\usepackage{amsmath}
\usepackage{bbm}
\usepackage{mathtools}
\usepackage{amssymb}
\usepackage{graphicx}
\usepackage{color}
\usepackage{latexsym}
\usepackage[utf8]{inputenc}
\usepackage[T1]{fontenc}
\usepackage{comment}
\usepackage{stmaryrd}

\newcommand{\ns}{{\mathbb N}} 
\newcommand{\qs}{{\mathbb Q}}  
\newcommand{\cs}{{\mathbb C}} 

\newcommand{\al}{\alpha}

\newcommand{\si}{\sigma}
\newcommand{\la}{\lambda}
\newcommand{\eps}{\varepsilon}

\newcommand{\bu}{\bar u}

\newcommand{\Sn}{{\mathfrak S}}

\newcommand{\GK}{\mathbb{K}}
\newcommand{\GL}{\mathbb{L}}

\DeclareMathOperator{\DR}{DR}

\DeclareMathOperator{\id}{id}
\DeclareMathOperator{\Park}{Park}
\DeclareMathOperator{\Tam}{Tam}

\newcommand{\tG}{\tilde G}
\newcommand{\tz}{\tilde z}

\newcommand{\cL}{\mathcal L}

\newcommand{\cT}{\mathcal T}

\newcommand{\p}{permutation}
\newcommand{\ps}{permutations}

\newcommand{\figeps}[3]
{\begin{figure}[ht!]
\begin{center} 
\includegraphics[width=#1cm]{#2.eps}\caption{#3}\label{fig:#2} 
\end{center}
\end{figure}}

\newtheorem{Theorem}{Theorem}
\newtheorem{Proposition}[Theorem]{Proposition}

\newtheorem{Corollary}[Theorem]{Corollary}

\newtheorem{Lemma}[Theorem]{Lemma}

\newtheorem{Definition}[Theorem]{Definition}

\newcommand{\beq}{\begin{equation}}
\newcommand{\eeq}{\end{equation}}

\newcommand{\gf}{generating function}
\newcommand{\gfs}{generating functions}
\newcommand{\fps}{formal power series}

\def\emm#1,{{\em #1}}

\graphicspath{{Figures/}}

\catcode`\@=11
\def\section{\@startsection{section}{1}%
 \z@{.7\linespacing\@plus\linespacing}{.5\linespacing}%
 {\normalfont\bfseries\scshape\centering}}

\def\subsection{\@startsection{subsection}{2}%
  \z@{.5\linespacing\@plus\linespacing}{.5\linespacing}%
  {\normalfont\bfseries\scshape}}

\def\subsubsection{\@startsection{subsubsection}{3}%
 \z@{.5\linespacing\@plus\linespacing}{-.5em}
  {\normalfont\bfseries\itshape}}
\catcode`\@=12

\def\cT{\mathcal{T}}
\def\cTn{\cT_n}

\newcommand{\spacebreak}
{\begin{displaymath} \triangleleft \; \lhd \;
\diamond \; \rhd \; \triangleright
  \end{displaymath}}

\begin{document}
\title
[The representation of the symmetric group on $m$-Tamari intervals]
{The representation of the symmetric group\\
on $m$-Tamari intervals}

\author[M. Bousquet-M\'elou]{Mireille Bousquet-M\'elou}
\author[G. Chapuy]{Guillaume Chapuy}
\author[L.-F. Préville-Ratelle]{Louis-François Préville-Ratelle}

\address{MBM: CNRS, LaBRI, UMR 5800, Universit\'e de Bordeaux, 
351 cours de la Lib\'eration, 33405 Talence Cedex, France}
\email{mireille.bousquet@labri.fr}
\address{GC: CNRS, LIAFA, UMR 7089, Universit\'e Paris Diderot - Paris 7, Case 7014,
75205 Paris Cedex 13, France}
\email{guillaume.chapuy@liafa.jussieu.fr}
\address{LFPR: LACIM, UQAM, C.P. 8888 Succ. Centre-Ville, Montréal H3C
  3P8, Canada}
\address{\vskip -5mm(Current address: Instituto de Matem\'atica y F\'{\i}sica, Universidad de
Talca, 2 norte 685, Talca, Chile)
}
\email{preville-ratelle@inst-mat.utalca.cl}
%

\thanks{GC was partially supported by the LIRCO and the European project
  ExploreMaps -- ERC StG 208471. LFPR was partially supported by
  the latter project, by
  a Canadian CRSNG graduate scholarship, an FQRNT ``stage
  international'', and finally  by CONICYT (Comisi\'on
  Nacional de Investigaci\'on Cient\'ifica y Tecnol\'ogica de Chile)
  via the Proyecto Anillo ACT56.}

%

\begin{abstract}
An $m$-ballot path of size $n$ is a path  on the square grid
consisting of north and east unit steps, starting at
$(0,0)$,  ending at $(mn,n)$, and never going below the line
$\{x=my\}$. The set of these paths can be equipped with a lattice structure,
called the $m$-Tamari lattice and denoted by $\cTn^{(m)}$, which
generalizes the usual Tamari  
lattice $\cTn$ obtained when $m=1$. 
This lattice was  
introduced by F.~Bergeron   in connection with the
study of diagonal coinvariant spaces in three sets of $n$
variables. 
 The representation of the symmetric group $\Sn_n$ 
on these spaces  is conjectured to be
closely related to  the natural  representation of $\Sn_n$ on  (labelled)  intervals of
the $m$-Tamari lattice, which we study in this paper.

An interval $[P,Q]$ of 
$\cTn^{(m)}$ is \emm labelled, if the north steps of $Q$ are labelled
from 1 to $n$ in such a way the labels increase along any sequence of
consecutive north steps.
The symmetric group $\Sn_n$ acts on labelled intervals of $\cTn^{(m)}$
by permutation of the labels.
We prove an explicit formula, conjectured by F.~Bergeron
and the third author,
for the character of the associated  representation of $\Sn_n$.
In particular,
the dimension of the representation, that is,
the number of labelled $m$-Tamari intervals of size $n$,
is found to be
$$
{(m+1)^n(mn+1)^{n-2}}.
$$
These results are new, even when $m=1$.

The form of these numbers suggests a connection with parking
functions, but our proof is not bijective. The starting point is a
recursive description of $m$-Tamari intervals. It yields an equation for an
associated generating function, which 
is   a refined  version of the Frobenius series of the
representation.
This equation
 involves two additional variables $x$ and $y$, a derivative with
respect to $y$ and iterated divided differences with respect to $x$.
The hardest part of the proof consists in solving it, and we develop
original techniques to do so,
partly inspired by previous work on polynomial equations with ``catalytic'' variables.
\end{abstract}

\date{today}
\maketitle

\section{Introduction and main result}
An \emph{$m$-ballot path} of size $n$ is a path  on the square grid
consisting of north and east unit steps, starting at
$(0,0)$,  ending at $(mn,n)$, and never going below the line
$\{x=my\}$. 
It is well-known that  there are 
$$\frac
1{mn+1}{(m+1)n \choose n}$$ such paths~\cite{dvoretzky}, and that they
are in bijection with $(m+1)$-ary trees with $n$ inner nodes.

François Bergeron recently defined on the  set $\cT_n^{(m)}$ of $m$-ballot
paths of size $n$ an order relation. It is 
convenient to  describe it via the associated covering relation,
 exemplified in Figure~\ref{fig:push_mWalk}.
\begin{Definition}
\label{def-m-tamari}
Let $P$ and $Q$ be two $m$-ballot paths of size $n$.
 Then $ Q$ covers $P$ if  there exists in  $P$
 an east step $a$, followed by a north step $b$, such that $Q$ is
 obtained from $P$ by swapping $a$ and $S$, 
 where $S$ is the shortest factor of $P$ that begins with $b$ and is
 a (translated) $m$-ballot path.
\end{Definition}

\figeps{12}{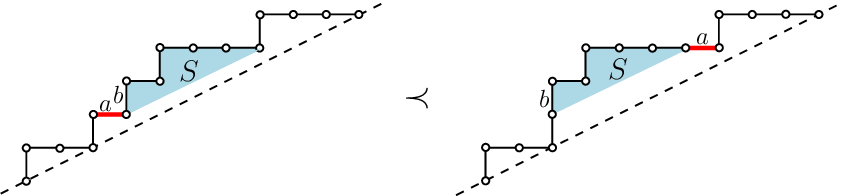}{The covering relation  between
  $m$-ballot paths ($m=2$).}

It was shown in~\cite{bousquet-fusy-preville} that
this order endows $\cT_n^{(m)}$ with a lattice
structure, which is called the \emm $m$-Tamari lattice of size
$n$,. When $m=1$,  it coincides with the classical Tamari
lattice $\cT_n$~\cite{BeBo07,friedman-tamari,HT72,knuth4}. Figure~\ref{fig:lattice_ex} 
shows two of the lattices $\cT_n^{(m)}$. 
The main result of~\cite{bousquet-fusy-preville} gives the number of
intervals in $\cT_n^{(m)}$ as
\beq\label{unlabelled}
\frac {m+1}{n(mn+1)} {(m+1)^2 n+m\choose n-1}.
\eeq
The lattices $\cT_n^{(m)}$ are also known to be EL-shellable~\cite{muehle}.

\begin{figure}[h]
\begin{center}
\includegraphics[height=9cm]{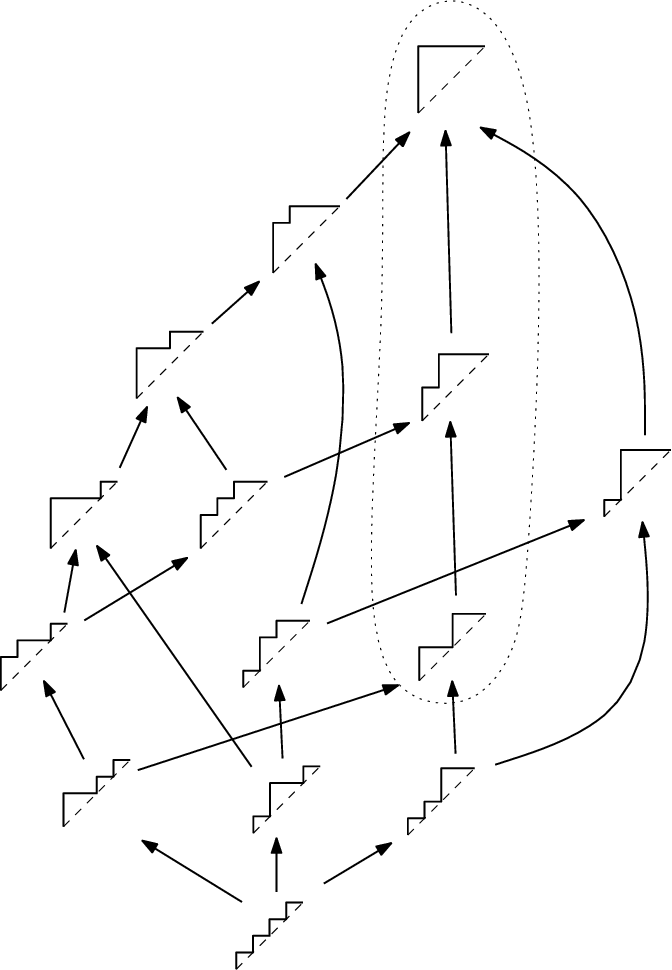}\hspace{10mm}\includegraphics[height=9cm]{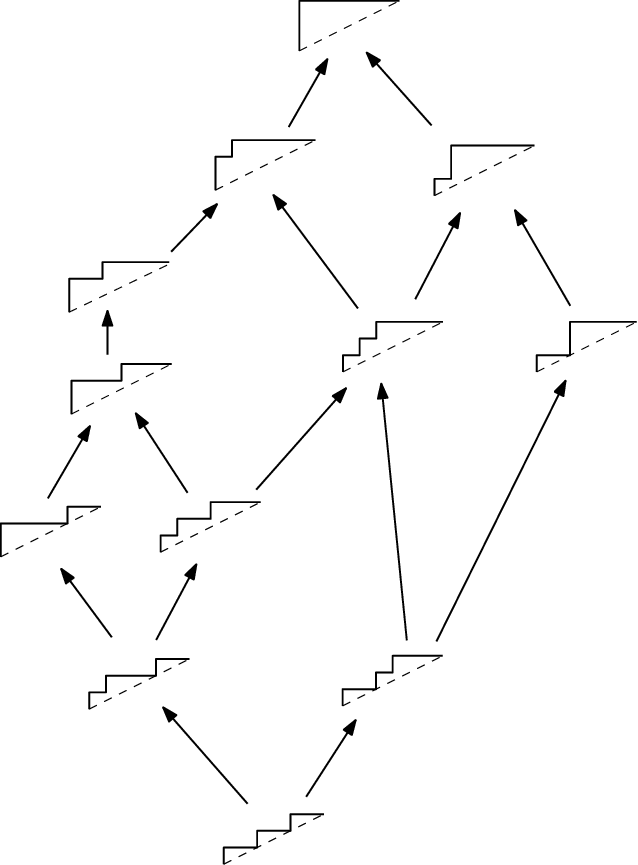}
\end{center}
\caption{The $m$-Tamari lattice $\cT_{n}^{(m)}$ for  $m=1$ and $n=4$ (left)
and for $m=2$ and $n=3$ (right). 
The three walks surrounded by a line in $\cT_{4}^{(1)}$ form a lattice
that is isomorphic to $\cT_{2}^{(2)}$ (see Proposition~\ref{prop:sublattice}).}
\label{fig:lattice_ex}
\end{figure} 

The interest in these lattices 
is motivated by their
--- still conjectural ---  connections with trivariate diagonal coinvariant
spaces~\cite{bergeron-preville,bousquet-fusy-preville}. Some of these connections
are detailed at the end of this introduction. In particular, 
it is believed that the representation of the symmetric group on these
spaces is closely related to the representation of the symmetric group
on \emm labelled, $m$-Tamari intervals. The aim of this paper is to
characterize the latter representation, by describing explicitly its
character. 

So let us define this representation and state our main result. 
Let us call \emm ascent, of a path a maximal sequence of consecutive
north steps. 
An $m$-ballot path of size 
$n$ is \emm labelled, if the north steps are labelled from 1 to $n$,
in such a way the labels increase along 
ascents
(see the upper paths in Figure~\ref{fig:tamari-action}).  These paths
are  in bijection with  $(1,m,\ldots, 
m)$\emm -parking functions, of size $n$, in the sense 
of~\cite{pitman-stanley,yan}: the function $f$ associated with a  path $Q$ satisfies
$f(i)=k$ if the north step of $Q$ labelled $i$ lies at abscissa
$k-1$. 
The symmetric group
$\Sn_n$ acts on labelled $m$-ballot paths of size $n$ by permuting 
labels, and then reordering them in each ascent
(Figure~\ref{fig:tamari-action}, top paths). The character of this representation,
evaluated at a permutation of cycle type $\la=(\la_1, \ldots,
\la_\ell)$, is 
$$
(mn+1)^{\ell-1}.
$$
This formula is easily
proved using  the cycle lemma~\cite{riordan}. As recalled further down,
this representation is closely related to the representation of
$\Sn_n$ on diagonal coinvariant spaces in two sets of variables.

{\begin{figure}[t!]
\begin{center} 
\includegraphics[width=14cm]{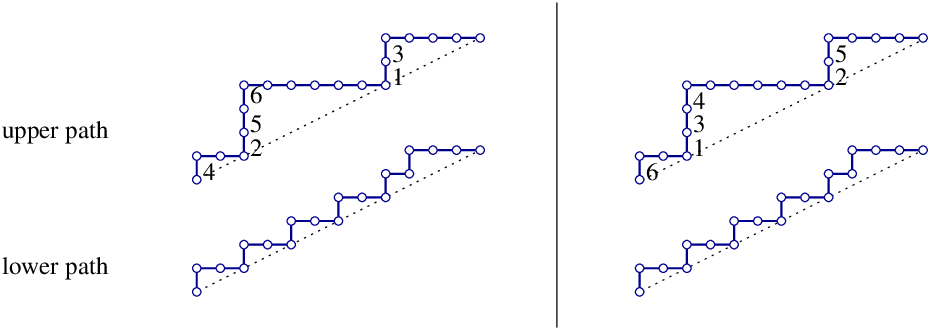}
\caption{A labelled $2$-Tamari interval, and its image under the
  action of $\sigma= 2\, 3\, 5\, 6\, 1\, 4$.}
\label{fig:tamari-action} 
\end{center}
\end{figure}}

Now an $m$-Tamari interval $[P,Q]$ is \emm
labelled, if the upper path $Q$ is labelled.  The symmetric group
$\Sn_n$ acts on labelled intervals of $\cT_n^{(m)}$ by rearranging
the labels of $Q$ as described above 
(Figure~\ref{fig:tamari-action}). 
We call this representation the \emm
$m$-Tamari, representation of $\Sn_n$. Our main result is an explicit
expression for its character 
$\chi_m$, which was
conjectured by Bergeron and the third author~\cite{bergeron-preville}.

\begin{Theorem}\label{thm:char}
Let $\la=(\la_1, \ldots, \la_\ell)$ be a partition of $n$ and
$\sigma$ a permutation of $\Sn_n$ having cycle type $\la$. 
Then for the $m$-Tamari representation of $\Sn_n$,
\beq\label{caractere}
\chi_m(\sigma)= (mn+1)^{\ell-2} \prod_{1 \leq i \leq \ell}
{\binom{(m+1) \lambda_i}{\lambda_i}}. 
\eeq
Since $\Sn_n$ acts by permuting labelled intervals, this is also  the number of
labelled $m$-Tamari intervals left unchanged under the action of $\sigma$.
The value of the character only depends on the cycle type $\la$, and
will sometimes be denoted $\chi_m(\la)$. 

In particular, the dimension of the representation, that is, the
number of labelled $m$-Tamari intervals of size $n$, is
\beq\label{dimension}
 \chi_m({\id})= (mn+1)^{n-2} (m+1)^n.
\eeq
 \end{Theorem}
 We were unable to find a bijective proof of
these amazingly simple formulas. 
Instead, our proof uses generating functions and a recursive
construction of intervals.  Our main \gf\ records the numbers
$\chi_m(\sigma)$, that is, the number of pairs $(I,\sigma)$ where $I$
is a labelled interval fixed by the \p\ $\sigma$. This \gf\ involves
variables $p_1, p_2, \ldots$ (keeping track of the cycle type of
$\sigma$) and $t$ (for the size of $I$). 
%
The recursive construction of intervals that we use is borrowed
from~\cite{bousquet-fusy-preville}. In order to translate it into an
equation defining our \gf, we need to keep track of one more 
parameter defined on $(I, \sigma)$, using  an additional variable $x$
(Proposition~\ref{prop:eq}, Section~\ref{sec:eq}). The resulting
equation involves discrete derivatives
(a.k.a. divided differences) with respect to $x$, of unbounded
order. The solution of equations with discrete derivatives  of
\emm bounded, order is now well-understood~\cite{mbm-jehanne} (such
equations are for instance involved 
 in the enumeration~\eqref{unlabelled} of unlabelled $m$-Tamari
intervals). But this is the first time we meet an equation of unbounded
order, and its solution is 
the most difficult and original part of the paper. Our approach requires to introduce one more variable $y$, and a
derivative with respect to it. Its principles   are explained  in
Section~\ref{sec:sol1}, and exemplified with the
case $m=1$. The general case is solved in Section~\ref{sec:sol}. 
This approach was already used in a preprint
by the same
authors~\cite{mbm-chapuy-preville}, where the special
case~\eqref{dimension} was proved. Since going from~\eqref{dimension}
to~\eqref{caractere} implies a further complexification, this preprint
may serve as an
introduction to our techniques. 
 The present paper is however
 self-contained.  
Section~\ref{sec:comments} gathers a few final comments. In
particular, we reprove 
the main 
 result of~\cite{bousquet-fusy-preville}
giving the number of \emm unlabelled, intervals of $\cT_n^{(m)}$.

\medskip

In the remainder of this section, we 
recall some of the conjectured connections
between Tamari intervals and trivariate diagonal coinvariant
spaces. They seem to parallel the (now largely proved) connections between
ballot paths and \emm bivariate, diagonal coinvariant
spaces, which have  attracted considerable attention in the past 20
years~\cite{MR1935784,MR1972636,haglund-book,MR2163448,HaiConj,HaimanPreu,loehr-thesis} and are
still a very active area of research
today~\cite{armstrong-arrangements,armstrong-tesler,garsia-hicks-stout,garsia-xin-zabrocki,haglund-touch-points,haglund-polynomial,hicks,lee}.

Let $X=(x_{i,j})_{^{1 \leq i \leq k}_{1 \leq j \leq n}}$ be a matrix
of variables.  The diagonal
coinvariant space $\mathcal{DR}_{k,n}$ is defined as the quotient of
the ring $\mathbb{C}[X]$  of polynomials in the coefficients of $X$ by the ideal
$\mathcal{J}$ generated by 
constant-term free polynomials that are invariant under permuting the columns of
$X$. For example, when $k=2$, denoting $x_{1,j}=x_j$ and $x_{2,j}=y_j$,
the ideal $\mathcal J$ is generated by constant-term free polynomials
$f$ such that for all $\sigma \in \Sn_n$,
$$
f(X)=\sigma(f(X)), \quad \hbox{where} \quad
\sigma(f(X))=f(x_{\sigma(1)}, \ldots, x_{\sigma(n)},y_{\sigma(1)},
\ldots, y_{\sigma(n)}).
$$
An \emph{m-extension} of the spaces $\mathcal{DR}_{k,n}$ is of great importance
here~\cite[p.~230]{garsia-haiman-lagrange}. Let $\mathcal{A}$ be the ideal
of $\cs[X]$ 
generated by \emm alternants, under the  diagonal action described
above;
that is, by polynomials $f$ such that $\sigma(f(X)) = \eps(\sigma) f(X)$.
There is a natural action of $\Sn_n$ on the quotient space
$\mathcal{A}^{m-1} \big{/} \mathcal{J} \mathcal{A}^{m-1}$. Let us
twist this action by the $(m-1)^{\rm st}$ power of the sign
representation $\eps$: this gives rise to spaces
$$
\mathcal{DR}_{k,n}^{m} := {\eps}^{m-1}  \otimes  \mathcal{A}^{m-1} \big{/} \mathcal{J} \mathcal{A}^{m-1},
$$
so that $\mathcal{DR}_{k,n}^{1}=\mathcal{DR}_{k,n}$.
It is now a famous theorem of Haiman~\cite{HaimanPreu,MR2115257} 
that,  as representations of $\Sn_n$,
$$
\mathcal{DR}_{2,n}^{m} \cong \eps \otimes \:   \Park_m(n)
$$
where $  \Park_m(n)$ is the
$m$-parking representation of $\Sn_n$, that is, the representation on
$m$-ballot paths of size $n$ defined above.

In the case of three sets of variables, 
Bergeron and Pr\'eville-Ratelle~\cite{bergeron-preville} conjecture  that, as 
representations of $\Sn_n$, 
$$
\mathcal{DR}_{3,n}^{m} \cong \eps \otimes  \Tam_m(n),
$$
 where $\Tam_m(n)$ is the $m$-Tamari representation of $\Sn_n$.
The fact that the dimension of this space seems to be  given
by~\eqref{dimension} is an 
earlier conjecture due to F.~Bergeron. This  was also observed 
earlier for small values of $n$ by Haiman~\cite{HaiConj} in the case $m=1$.

\section{The refined Frobenius series}
\label{sec:frobenius}
\subsection{Definitions and notation}
\label{sec:notation}

Let $\GL$ be a commutative ring and $t$ an indeterminate. We denote by
$\GL[t]$ 
(resp. $\GL[[t]]$) 
the ring of polynomials 
(resp. \fps) 
in $t$ with coefficients in $\GL$. If $\GL$ is a field, then $\GL(t)$
denotes the field 
of rational functions in $t$.
This notation is generalized to polynomials, fractions
and series in several indeterminates. We 
denote by bars the reciprocals of variables: for instance, $\bu=1/u$, so that
$\GL[u,\bu]$ is the ring of Laurent
polynomials in $u$ with coefficients in $\GL$.
The coefficient of $u^n$ in a Laurent polynomial $P(u)$ is denoted
by $[u^n]P(u)$.

We use classical notation relative to integer partitions,
 which we recall briefly. A \emm partition, $\lambda$ of $n$ is
a non-increasing sequence of integers $\lambda_1\geq \lambda_2\geq \dots \geq
\lambda_\ell>0$ summing to $n$. We write $\lambda \vdash n$ to mean that $\lambda$
is a partition of $n$. Each component $\la_i$ is called a \emm
part,. The number of parts  or \emm length, of the
partition is denoted by $\ell(\lambda)$. 
The \emph{cycle type} of a permutation $\sigma \in \mathfrak{S}_n$ is the
partition of $n$ whose parts are the lengths of the cycles of $\sigma$. 
This partition is denoted by $\lambda(\sigma)$.
The number of permutations $\sigma\in\mathfrak{S}_n$ having 
cycle type $\lambda\vdash n$ equals $\frac{n!}{z_\lambda}$
 where $z_\lambda:=\prod_{i\geq 1} i^{\al_i}\al_i!$,
where $\al_i$ is the number of  parts equal to $i$ in $\lambda$.

We let $p=(p_1,p_2,\dots)$ be an infinite list of independent variables, and for
$\lambda$ a partition of $n$, we let $p_\lambda =
p_{\lambda_1}\cdots p_{\lambda_{\ell(\lambda)}}$. The reader may view
the $p_\lambda$'s as 
power sums in some ground set of variables 
(see {\emph e.g.}~\cite{sagan-book}). This point of view is not really
needed in this paper, but it explains why
we call our main generating function a \emm refined Frobenius series,.
Throughout the paper, we denote by
$\mathbb{K}=\mathbb{Q}(p_1,p_2,\dots)$ the field of rational fractions 
in the~$p_i$'s with rational coefficients.

\medskip

 Given a Laurent polynomial $P(u)$ 
in a variable $u$, we denote by $[u^\ge]P(u)$ the \emm non-negative part of
$P(u)$ in $u$,, defined by
$$
[u^\ge]P(u) = \sum_{i\ge 0} u^i \, [u^i] P(u) .
$$
The definition is then extended by linearity to power series
whose coefficients are Laurent polynomials in $u$. We define similarly the
positive part of $P(u)$, denoted by $[u^>]P(u)$.

\medskip

We now introduce several series and polynomials which 
play an important 
role in this
paper. They depend on two independent variables $u$ and $z$.
First, we let $v\equiv v(u)$ be the following  Laurent polynomial in $u$:
$$
v= (1+u)^{m+1}u^{-m}.
$$
We now consider the following series:
\begin{eqnarray}\label{eq:defV}
  V(v) = \sum_{k\geq 1} \frac{p_k}{k} v^k z^k.
\end{eqnarray}
It is is a formal power series in $z$ whose coefficients are Laurent
polynomials in $u$ over the field $\mathbb{K}$.
Finally we define the two following formal power series in $z$:
\begin{eqnarray}
L\equiv L(z,p)&:=&[u^0] V(v)
=\sum_{k\geq 1} \frac{p_k}{k} 
{(m+1)k\choose k}
z^k,
\label{eq:defL} \\
K(u)\equiv K(z,p;u) &:=&[u^>] V(v)=\sum_{k\geq 1} \frac{p_k}{k}z^k
\sum_{i=1}^{k} {(m+1)k\choose {k-i}} u^{i}. \label{eq:defK}
\end{eqnarray}
As shown with these series, we often do not denote the dependence of our series
in certain variables (like $z$ and $p$ above). This is indicated by
the symbol $\equiv$.

\subsection{A refined theorem}

As stated in Theorem~\ref{thm:char}, the
value of the  character  $\chi_m(\sigma)$ is 
the number of labelled intervals fixed  under
the action of $\sigma$, and  one may see~\eqref{caractere} as an
enumerative result.  Our main result is  a refinement
of~\eqref{caractere} where we 
take into account two more parameters, which we now
define. The first parameter is the number of \emm contacts, of the interval:
 A \emph{contact} of a ballot path $P$ is a vertex of $P$ lying
on the 
line $\{x=my\}$,
 and a \emm contact, of a Tamari interval $[P, Q]$ is a contact of the
 \emm lower, path $P$. We denote by $c(P)$ the number of contacts of $P$.

By definition of the action of $\mathfrak{S}_n$ on $m$-Tamari
intervals, a labelled interval 
$I=[P,Q]$ is 
fixed by a permutation $\sigma \in \mathfrak{S}_n$ if and
only if $\sigma$ stabilizes the set of labels 
of each ascent of $Q$. 
Equivalently, each cycle of $\sigma$ is contained in 
the label set of 
an ascent of $Q$.
If this holds, we let 
$ a_\sigma(Q)$ be the number of cycles of $\sigma$ 
occurring in the \emm first, ascent of $Q$: this is our second parameter.

The main object we handle in this paper is a 
generating function for pairs $(\sigma, I)$, where
$\sigma$ is a permutation and $I=[P,Q]$ is a labelled $m$-Tamari interval fixed by
$\sigma$. In this series $F^{(m)}(t,p;x,y)$, pairs $(\sigma, I)$ are
counted 
by the size $|I|$ of $I$ (variable $t$), the number $c(P)$ of contacts (variable $x$), the
parameter $a_\sigma(Q)$ (variable $y$), and the cycle type of $\sigma$
(one variable $p_i$ for each cycle of size $i$ in~$\sigma$). Moreover,
$F^{(m)}(t,p;x,y)$ is an exponential series in $t$. That is,
 \beq\label{F-def}
 F^{(m)}(t,p;x,y)= 
\sum_{I=[P,Q], \ \rm{\small{labelled}}}\frac{ t^{|I|}}{|I|!}\,  x^{c(P)}
 \sum_{\sigma \in \mathrm{Stab}(I)} 
 \,y^{a_\sigma(Q)} p_{\lambda(\sigma)},
\eeq
where the first and second  sums are taken respectively over all labelled
$m$-Tamari intervals $I$,  and over all permutations
 $\sigma$ fixing $I$.

Note that when $(x,y)=(1,1)$, we have:
$$ 
F^{(m)}(t,p;1,1)= 
\sum_{I=[P,Q]}\frac{ t^{|I|}}{|I|!} 
 \sum_{\sigma \in \mathrm{Stab}(I)} p_{\lambda(\sigma)}
= \sum_{n\geq 0} \frac{ t^n}{n!} \sum_{\sigma \in \Sn_n} \chi_m(\si)
 {p_{\lambda(\si)}}
= \sum_{n\geq 0} t^{n} \sum_{\lambda \vdash n} \chi_m(\lambda)
 \frac{p_\lambda}{z_\lambda},
$$
since the number of intervals fixed by a permutation 
depends only on its cycle type,
and since $\frac{n!}{z_\lambda}$ is the number of permutations of cycle type
$\lambda$.
Hence, in representation theoretic terms,  $[t^n]F^{(m)}(t,p;1,1)$ is 
the \emm Frobenius characteristic, of the $m$-Tamari 
representation of $\Sn_n$,
also equal to
$$
\sum_{\la\vdash n} c(\la) s_\la,
$$
where $s_\la$ is the Schur function of shape $\la$ and $c(\la)$ is the
multiplicity of the irreducible representation associated with $\la$
in the $m$-Tamari representation~\cite[Chap.~4]{sagan-book}.
For this reason, we call $F^{(m)}(t,p;x,y)$ a \emm refined
Frobenius series,. 

The most general  result of this paper is a (complicated) 
parametric expression of $F^{(m)}(t,p;x,y)$,
which takes the following simpler form when $y=1$.
\begin{Theorem}\label{thm:main}
Let  $F^{(m)}(t,p;x,y)\equiv F(t,p;x,y)$ be the refined Frobenius
series of the $m$-Tamari representation, 
defined by~\eqref{F-def}. Let $z$ and
$u$ be two indeterminates, and write 
\beq\label{t-x-param}
t=z e^{-mL}
\quad \hbox{and } \quad x=({1+u})e^{-mK(u)},
\eeq
where $L\equiv L(z,p)$ and $K(u)\equiv K(z,p;u)$ are defined by
\eqref{eq:defL} and \eqref{eq:defK}.
Then $F(t,p;x,1)$ becomes a series in $z$ with polynomial coefficients
in $u$ and the $p_i$, and this series has a simple expression:
\beq\label{Fx1}
F(t,p;x,1)= ({1+\bu})e^{K(u)+L} \left((1+u)e^{-mK(u)} - 1\right)
\eeq
with $\bu=1/u$. In particular, 
in the limit $u\rightarrow 0$, we obtain
\beq \label{Fx1-2}
F(t,p;1,1)=
e^{L}\left(1-m\sum_{k\geq1}\frac{p_k}{k}z^k
{(m+1)k\choose k-1}
\right).
\eeq
\end{Theorem}
\noindent 
The form of this theorem is reminiscent of the enumeration of
unlabelled $m$-Tamari intervals~\cite[Thm.~10]{bousquet-fusy-preville}, for which finding the ``right''
parametrization of the variables $t, x$ and $y$ was 
an important step in the
solution. This will also be the case here.

Theorem~\ref{thm:char} will follow from Theorem~\ref{thm:main} by extracting 
the coefficient of $p_\la/z_\la$ 
in $F(t,p;1,1)$ (via Lagrange's inversion).
Our expression of $F^{(m)}(t,p;x,y)$ is given in
Theorem~\ref{thm:trivariate}. When 
$m=1$, it takes a reasonably simple form, which we now present. The
case $m=2$ is also detailed 
at the end of Section~\ref{sec:sol} (Corollary~\ref{coro:m=2}).
\begin{Theorem}\label{thm:1}
Let  $F^{(1)}(t,p;x,y)\equiv F(t,p;x,y)$ be the refined Frobenius
series of the $1$-Tamari representation, 
defined by~\eqref{F-def}. 
Define the series $V(v), L$ and $K(u)$ by
{\rm{(\ref{eq:defV}--\ref{eq:defK})}}, with $m=1$, and perform the change of
variables~\eqref{t-x-param}, still with $m=1$.
Then $F(t,p;x,y)$ becomes a formal power series in $z$ with polynomial coefficients
in $u$ and $y$,  which is given by
\beq\label{F-param-y1}
F(t,p;x,y)= (1+u)\ [u^{\ge}] \Big(e^{yV(v)-K(u)} - \bu e^{yV(v)-K(\bu)} \Big),
\eeq
with $\bu=1/u$.
\end{Theorem}
\noindent{\bf Remarks}\\
1. It is easily seen that the case $y=1$ of~\eqref{F-param-y1} reduces to the case $m=1$ of~\eqref{Fx1} (the proof relies on
the fact that 
$L$ and $K(u)$ are respectively the constant term and the positive
part of $V(v)$ in $u$,
and that $v=(1+u)(1+\bu)$ is symmetric in $u$ and $\bu$). 

\smallskip
\noindent 2. When $p_1=1$ and $p_i=0$ for $i>1$, the only permutation
that contributes in~\eqref{F-def} is the identity. We are thus simply counting labelled
$1$-Tamari intervals, by their size (variable $t$), the number of
contacts (variable $x$) and the size of the first ascent (variable
$y$). 
Still taking $m=1$, 
we have 
$V(v)=zv=z(1+u)(1+\bu)$, 
$K(u)=zu$ and the extraction of the positive part in $u$ in~\eqref{F-param-y1}
can be performed explicitly:
\begin{eqnarray*}
  {F(t,p;x,y)}&=& (1+u) [u^{\ge}] \left( e^{yzv-zu}-\bu
    e^{yzv-z\bu}\right)
\\
&=&(1+u)e^{2yz}\left(\sum_{0\le i \le j} u^{j-i} \frac{z^{i+j}y^i (y-1)^j}{i! j!}
- 
\sum_{0\le j < i} u^{i-j-1} \frac{z^{i+j}y^i (y-1)^j}{i! j!}\right).
\end{eqnarray*}
When $x=1$, that is, $u=0$, the double
sums in this expression reduce to simple sums, and the
\gf\ of labelled Tamari intervals, counted by the size and the height
of the first ascent, is expressed in terms of Bessel functions:
$$
\frac{F(t,p;1,y)}{ e^{2yz}}= \sum_{i\ge 0 }  \frac{z^{2i}y^i (y-1)^i}{i!^2}
- 
\sum_{ j\ge 0 } \frac{z^{2j+1}y^{j+1} (y-1)^j}{(j+1)! j!}.
$$

\section{A functional equation}
\label{sec:eq}

The aim of this section is to establish a functional equation  satisfied
by the series $F^{(m)}(t,p;x,y)$. 
\begin{Proposition}\label{prop:eq}
For $m\ge 1$, let  $ F^{(m)}(t,p;x,y)\equiv F(x,y)$ be the refined
Frobenius series of the $m$-Tamari representation, defined by~\eqref{F-def}.
Then 
\begin{eqnarray*}
  \label{eq:F}
  F(x,y)&= &\sum_{k\geq 0} {\tilde h_k(y)}
  \left(tx(F(x,1)\Delta)^{(m)}\right)^{(k)} (x)
\\
&=&
  \exp\left(y \sum_{k\ge 1}\frac{p_k}{k} \left( tx
  (F(x,1)\Delta)^{(m)}\right)^{(k)}\right) (x),
\end{eqnarray*}
where 
\beq\label{htdef}
\tilde h_k(y) = \sum\limits_{\lambda\vdash
  k}\frac{p_\lambda}{z_\lambda}\, y^{\ell(\lambda)},
\eeq
$\Delta$ is the following divided difference operator
$$
\Delta S(x)=\frac{S(x)-S(1)}{x-1},
$$
and the powers $(m)$ and $(k)$ mean that the operators
are applied respectively $m$ times 
and $k$ times.

Equivalently,
$F(x,0)=x$ and
\beq\label{eq:Fb}
\frac{\partial F}{\partial y} (x,y)=
\sum_{k\geq 1} \frac{p_k}{k} \left(tx(F(x,1)\Delta)^{(m)}\right)^{(k)}
(F(x,y)).
\eeq
\end{Proposition}
The above equations rely on  a recursive 
construction of labelled
$m$-Tamari intervals. Our description of the
construction is  self-contained, but we refer to~\cite{bousquet-fusy-preville} 
for several proofs and details. 

\subsection{Recursive construction of Tamari intervals}

We start by modifying the appearance of 1-ballot paths. 
We apply a 45 degree rotation
to transform them into \emm
Dyck paths,.  A Dyck path of size $n$ consists  of steps $u=(1,1)$ (up steps) and
steps $d=(1,-1)$ (down steps), starts at $(0,0)$, ends at $(2n,0)$ and never goes
below the $x$-axis.  We say that an up step has \emm rank, $i$ if it is the
$i^{\hbox{\small th}}$ up step of the path. We often represent Dyck
paths by words on the alphabet $\{u,d\}$.
An ascent is thus now a maximal sequence of $u$ steps. 

Consider  an $m$-ballot path of size $n$, and  replace each north step
 by a sequence of $m$ north steps. This gives a 1-ballot path of size $mn$, and
 thus, after a rotation, a Dyck path. In this path, for each
 $i\in\llbracket 0,n-1\rrbracket$, 
the up steps of ranks $mi+1,\ldots,m(i+1)$
are consecutive. We call the Dyck paths satisfying this property
\emph{$m$-Dyck paths}, and say that the up steps
of ranks $mi+1,\ldots,m(i+1)$ form a \emph{block}. Clearly, $m$-Dyck
paths of size $mn$ (\emm i.e.,, having $n$ blocks) are in
one-to-one correspondence with $m$-ballot paths of size $n$.

We often denote by $\cT_n$, rather than $\cT_n^{(1)}$, the usual
Tamari lattice of size $n$. Similarly, the intervals of this lattice are called
Tamari intervals rather than 1-Tamari intervals.
As proved in~\cite{bousquet-fusy-preville}, the 
transformation of $m$-ballot paths into $m$-Dyck paths maps $\cT_n^{(m)}$ on a sublattice of $\cT_{mn}$.
\begin{Proposition}[{\cite[Prop.~4]{bousquet-fusy-preville}}]
\label{prop:sublattice}
The set of $m$-Dyck paths with $n$ blocks is the sublattice of
$\mathcal{T}_{nm}$ consisting of the paths that are larger than or
equal to $u^m d^m \ldots u^m d^m$. It is
order isomorphic to $\mathcal{T}_{n}^{(m)}$. 
\end{Proposition}

We now describe a recursive construction 
 of (unlabelled) Tamari intervals,
again borrowed from~\cite{bousquet-fusy-preville}.
Thanks to the embedding of $\cTn^{(m)}$ into
$\mathcal{T}_{nm}$, it will also enable us to describe recursively
 $m$-Tamari intervals, for any value of $m$, in the next subsection.

A Tamari interval $I=[P,Q]$ is \emph{pointed} if its lower path $P$ has a
distinguished contact.
Such a contact splits $P$ into two Dyck paths
$P^\ell$ and $ P^r$, respectively located to the left and to the right
of the contact.
 The pointed interval $I$ is \emph{proper} 
if $P^\ell$ is not empty,
\emm i.e.,, if the distinguished contact is not $(0,0)$.
We often use the notation $I=[P^\ell P^r, Q]$ to denote a pointed
Tamari interval. 
The contact $(0,0)$ is called the \emm initial, contact.
\begin{Proposition}\label{prop:decomp}
Let $I_1=[P_1^\ell P_1^r,Q_1]$  be a pointed Tamari interval, and let
$I_2=[P_2,Q_2]$ be a Tamari interval.
Construct the  Dyck paths 
$$
P=uP_1^\ell dP_1^rP_2 \quad \mbox{ and } \quad    Q=uQ_1dQ_2 
$$
as shown in Figure~{\rm\ref{fig:concatenation}}.
Then $I=[P,Q]$ is a Tamari
interval. Moreover, the mapping $(I_1,I_2)\mapsto I$ is a bijection between
pairs $(I_1,I_2)$ formed of a pointed Tamari interval 
 and a Tamari
interval,
and Tamari intervals $I$ of positive size.
Note that $I_1$ is proper if and only if the 
 first ascent of $P$ has height larger than $1$.
\end{Proposition}
\begin{figure}
\begin{center}
\input{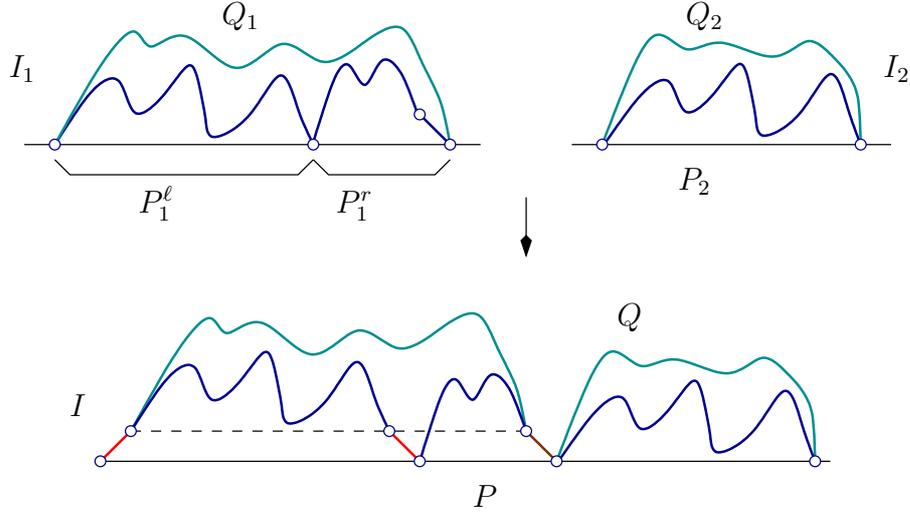}
\end{center}
\caption{The recursive construction of Tamari intervals.}
\label{fig:concatenation}
\end{figure}
\noindent{\bf Remarks}\\
\noindent 1. To recover $P_1^\ell$, $ P_1^r$, $Q_1$, $P_2$ and $Q_2$
from $P$ and $Q$, one proceeds as follows: $Q_2$ is the part of $Q$
that follows the first return of $Q$ to the $x$-axis; this also defines
$Q_1$ unambiguously. The path $P_2$ is the
suffix of $P$ having the same size as $Q_2$.  This also defines $P_1:=uP_1^\ell dP_1^r$
unambiguously. Finally, $P_1^r$ is the part of $P_1$
that follows the first return of $P_1$ to the $x$-axis, and this also
defines $P_1^\ell$ unambiguously.
\\
\noindent 2. This proposition is obtained  by combining Proposition~5
in~\cite{bousquet-fusy-preville} and
the case $m=1$ of Lemma~9 in~\cite{bousquet-fusy-preville}. With the
notation $(P';p_1)$ and $(Q',q_1)$ used 
therein, the above paths $P_2$
and $Q_2$ are respectively the parts of $P'$ and 
$Q'$ that lie to the right of $q_1$, while $P_1^\ell  P_1^r$ and
$Q_1$ are the parts of $P'$ and 
$Q'$ that lie to the left of $q_1$. The pointed vertex $p_1$  is the
endpoint of $P_1^\ell$. Proposition~5
in~\cite{bousquet-fusy-preville} guarantees that, if $P\preceq Q$ in
the Tamari order, then $P_1^\ell  P_1^r \preceq Q_1$ and $P_2 \preceq Q_2$.
\\
\noindent 3. One can keep track of several parameters in the
construction of Proposition~\ref{prop:decomp}. For instance,
the number of non-initial 
contacts of $P$ is 
\beq\label{eq:contacts}
c(P)-1=\left(c(P_1^r )-1\right)+c(P_2).
\eeq

\subsection{From the construction to the functional equation}

We now prove  Proposition~\ref{prop:eq} through a sequence of lemmas.
The first one describes $F^{(m)}(t,p;x,y)$ in terms of homogeneous
symmetric functions rather than power sums.
\begin{Lemma}\label{lem:ordinary}
  Let $\tilde h_k(y)$ be defined by~\eqref{htdef}, and set 
$$
h_k=\tilde h_k(1)=\sum\limits_{\lambda\vdash
  k}\frac{p_\lambda}{z_\lambda}.
$$
Hence $h_k$ is the $k^{\hbox{\small th}}$ homogenous symmetric function if
$p_k$ is the $k^{\hbox{\small th}}$ power sum. Then the refined Frobenius
series $F^{(m)}(t,p;x,y)$, defined by~\eqref{F-def}, can also be
written as the following \emm ordinary, \gf:
\beq\label{F:unlabelled}
F^{(m)}(t,p;x,y)= \sum_{I=[P,Q], \mbox{\small{ \emm unlabelled,}}} t^{|I|}x^{c(P)}\tilde
h_{a_1}(y) \prod_{i\ge 2} h_{a_i},
\eeq
where the sum runs over \emm unlabelled, $m$-Tamari intervals $I$, and
$a_i$ is the height of the $i^{\hbox{\small th}}$ ascent of the upper path $Q$.
In particular, $F^{(m)}(t,p;x,1)\equiv F^{(m)}(x,1)$ is the ordinary
\gf\ of $m$-Tamari intervals, counted by the size ($t$), the number of
contacts ($x$), and the distribution of ascents ($h_i$ for each
ascent of height $i$ in the upper path).
\end{Lemma}
\begin{proof}
 Let $I=[P,Q]$ be an unlabelled Tamari interval, and let $a_i$ be the
 height of the $i^{\hbox{\small th}}$ ascent of  $Q$. Denote
 $n=|I|$. An \emm $I$-partitioned permutation, is a permutation
 $\sigma \in \Sn_n$, together with a partition of the set of cycles of
 $\sigma$ into labelled subsets $A_1, A_2, \ldots$, such that the sum of
 the lengths of the cycles of $A_i$ is $a_i$. In the
 expression~\eqref{F-def} of $F^{(m)}$, the contribution of labelled intervals $\bar I=[P, \bar Q]$
 obtained by labelling $Q$ in all possible ways is
$x^{c(P)} \phi(I)$, where
$$
\phi(I)= \frac{t^{|I|} }{|I|!}  
 \sum_{\bar I=[P, \bar Q]} \sum_{\sigma \in \mathrm{Stab}(\bar I)} 
 \,y^{a_\sigma(\bar Q)} p_{\lambda(\sigma)}.
$$
In other words, $\phi(I)$ is the exponential \gf\ of $I$-partitioned permutations, counted by
the size (variable $t$), the number of cycles in the block $A_1$
(variable $y$), and
the number of cycles of length $j$ (variable $p_j$), for all $j$. By
elementary results on exponential \gfs, this series factors over
ascents of $Q$. The contribution of the $i^{\hbox{\small th}}$ ascent is the
exponential \gf\ of \ps\ of $\Sn_{a_i}$, counted by the size, the
number of cycles of length $j$ for all $j$, and also by the number of cycles
if $i=1$. But this is exactly $t^{a_i}h_{a_i}$ (or $t^{a_1}\tilde h_{a_1}(y)$ if
$i=1$), since
$$
t^{a}\tilde h_a(y)= {t^{a} }\sum\limits_{\lambda\vdash
  a}\frac{p_\lambda}{z_\lambda}\, y^{\ell(\lambda)}
            = \frac{t^{a} }{a!}\sum\limits_{\si \in
              \Sn_a}{p_\lambda(\si)}\, y^{\ell(\lambda(\si))}.
$$
Hence
$$
\phi(I)= t^{|I|} \tilde h_{a_1}(y) \prod_{i\ge 2} h_{a_i},
$$
and the proof is complete\footnote{
  An analogous result was used without proof in the study
  of the parking representation of the symmetric group~\cite[p.~28]{HaiConj}.}.
\end{proof}

\begin{Lemma}\label{lemma:H}
In the expression~\eqref{F:unlabelled} of $F^{(m)}(t,p;x,y)\equiv F(x,y)$, the contribution of
intervals $I=[P,Q]$ such that the first ascent of $Q$ has height $k$ is
$$
\tilde h_k(y) \left(tx(F(x,1)\Delta)^{(m)}\right)^{(k)} (x).
$$
This proves the first equation satisfied by $F^{(m)}(x,y)$ in Proposition~\rm{\ref{prop:eq}}.
\end{Lemma}

\begin{proof}
We constantly use  in this proof the inclusion $\mathcal{T}_n^{(m)}\subset
\mathcal{T}_{nm}$ given by Proposition~\ref{prop:sublattice}. That is, we
identify elements of $\mathcal{T}_n^{(m)}$ with $m$-Dyck paths having
$n$ blocks. 
The size of an interval is thus now the number of blocks, and the
height of the first ascent becomes the number of blocks in the first
ascent.

Lemma~\ref{lemma:H} relies on the recursive description of
Tamari intervals given in Proposition~\ref{prop:decomp}. 
We  actually apply this construction  to a slight
 generalization of $m$-Tamari intervals.
For $\ell \geq 0$, an \emph{$\ell $-augmented $m$-Dyck path} is a Dyck
path $Q$ of 
size $\ell +mn$ for some integer $n$, where the first $\ell $ steps
are up steps, and  all the other up steps can be partitioned into
\emph{blocks} of $m$  consecutive up steps. The  $\ell $ first steps
of $Q$
are not considered to be part of a block,   even if $\ell $ is a
multiple of $m$. 
We denote by
$a(Q)$  the number of blocks contained  in the first 
ascent\footnote{Since the number of blocks does not depend on $Q$
  only, but also on $\ell$, it should  in principle be denoted
  $a^{(\ell)}(Q)$. We hope that our choice of a lighter notation will
  not cause any confusion.}
   of $Q$.
A Tamari interval $I=[P,Q]$ is an \emph{$\ell $-augmented $m$-Tamari
  interval} if both $P$ and $Q$ are $\ell $-augmented $m$-Dyck paths.

\smallskip

For $k,\ell \geq0$ let $F_{k,\ell}(x)^{(m)}\equiv F_{k,\ell}(x)$ be the generating 
function of $\ell $-augmented $m$-Tamari intervals $I=[P,Q]$ such that
$a(Q)=k$, counted by the
number of blocks (variable $t$), the number of \emph{non-initial} contacts
(variable $x$) and the
number of non-initial ascents of $Q$ 
having $i$ blocks (one variable $h_i$ for each
$i\geq 1$, as before).
We are going to prove that for all $k,\ell \geq 0$, 
\beq\label{eq:Hk}
F_{k,\ell }(x)=
\left\{
\begin{array}{rl}
  \frac{1}{x}\left(tx(F(x,1)\Delta)^{(m)}\right)^{(k)} (x) &\mbox{ if } \ell =0,
  \\
  (F(x,1)\Delta)^{(\ell) }\left(tx(F(x,1)\Delta)^{(m)}\right)^{(k)} (x)& \mbox{ if } \ell >0.
\end{array}  \right.
\eeq
We claim  that \eqref{eq:Hk} implies Lemma~\ref{lemma:H}. Indeed,
$m$-Tamari intervals coincide with $0$-augmented
$m$-Tamari intervals. Since the initial
contact and the 
first ascent
 are not counted in $F_{k,0}(x)$, but are
counted in $F^{(m)}(x,y)$, the
contribution in $F^{(m)}(x,y)$ of intervals such that 
$a(Q)=k$  is $x \tilde h_k(y) F_{k,0}(x)$, 
as stated in the lemma.

\medskip
We now prove \eqref{eq:Hk}, by lexicographic induction on $(k,\ell )$. 
For $(k,\ell )=(0,0)$, the unique interval involved in $F_{k,\ell}(x)$ 
is $\{\bullet\}$, where $\bullet$ is the path
of length $0$. Its contribution is $1$ (since the initial and only contact is not
counted). Therefore $F_{0,0}(x)=1$ and~\eqref{eq:Hk} holds.
Let $(k,\ell )\neq (0,0)$ and assume 
that~\eqref{eq:Hk} holds for all
lexicographically smaller pairs  $(k',\ell ') < (k,\ell )$. We are
going  to show that~\eqref{eq:Hk} holds at rank $(k,\ell )$.

If $k>0$ and $\ell =0$, then we are considering $0$-augmented $m$-Tamari intervals, that is,
usual $m$-Tamari intervals. But an $m$-Tamari interval $I=[P,Q]$
having $n$ blocks and  $k$ blocks in the first ascent can be seen as an
$m$-augmented $m$-Tamari interval having  $n-1$ blocks and  $k-1$
blocks in the first ascent, by considering that
the first $m$ up steps of $P$ and $Q$ are not part of a block. This
implies that: 
$$
F_{k, 0} (x) = t F_{k-1,m}(x)= \frac{1}{x}\left(tx(F(x,1)\Delta)^{(m)}\right)^{(k)} (x) 
$$
by the induction hypothesis~\eqref{eq:Hk} applied at rank $(k-1,m)$. This is exactly~\eqref{eq:Hk} at rank $(k,\ell =0)$.

Now assume  $\ell \neq 0$. The series $F_{k,\ell }(x)$ 
counts $\ell $-augmented
$m$-Tamari intervals $I=[P,Q]$ such that $a(Q)=k$. By
Proposition~\ref{prop:decomp}, such an interval can be  decomposed into a
pointed interval $I_1=[P_1^\ell P_1^r,Q_1]$ and an interval
$I_2=[P_2,Q_2]$ 
(the ``$\ell$'' in the notation $P_1^\ell$ is a bit unfortunate here;
we hope it will not raise any difficulty). 
Note that $I_2$ is an $m$-Tamari interval, while $I_1$ is an $(\ell
-1)$-augmented pointed
$m$-Tamari interval. 
Conversely, starting from such a pair $(I_1, I_2)$, the construction of
Proposition~\ref{prop:decomp} produces an $\ell $-augmented
$m$-Tamari interval, unless $I_1$ is not proper and $\ell >1$.
Moreover, $a(Q_1)=a(Q)$.
Using~\eqref{eq:contacts}, we obtain:
\beq\label{eq:decomp10}
F_{k,\ell }(x) =  
F(x,1)\left(F^\bullet_{k,\ell -1}(x) + \mathbbm{1}_{\ell =1}F^\circ_{k,\ell -1}(x)\right)
\eeq
where $F^\bullet_{k,\ell -1}(x)$
(resp. $F^\circ_{k,\ell -1}(x)$) 
is the generating function of proper (resp. non-proper) pointed
$(\ell -1)$-augmented $m$-Tamari intervals  $I_1=[P_1^\ell P_1^r,Q_1]$ such that
$a(Q_1)=k$, counted by the number of blocks (variable $t$),
the number of non-initial ascents of $Q_1$ of height $i$ (variable
$h_i$) for each $i\geq 1$,
 and the number of non-initial contacts 
of $P_1^r$ (variable $x$). The factor $F(x,1)$ in~\eqref{eq:decomp10}
is the contribution of the interval $I_2$.

To determine the series $F^\bullet_{k,\ell -1}(x)$, expand the series
$F_{k,\ell -1}(x)$ as
$$
F_{k,\ell -1}(x) = \sum_{i\geq1} F_{k,\ell -1,i} x^i,$$
where $F_{k,\ell -1,i}=[x^i]F_{k,\ell -1}(x)$ is the generating function of
$(\ell -1)$-augmented $m$-Tamari intervals $[P_1,Q_1]$ such that $r(Q_1)=k$, and  having $i$
non-initial contacts. Each
such interval
can be pointed in $i$ different
ways to  give rise to $i$ different proper pointed intervals
$[P_1^\ell P_1^r,Q_1]$,
having respectively $0,1,\ldots, i-1$ non-initial contacts. 
Therefore,
\begin{eqnarray}
 F^\bullet_{k,\ell -1}(x) &=& \sum_i F_{k,\ell -1,i}(1+x+\dots+x^{i-1})
 \nonumber\\
&=& \sum_i F_{k,\ell -1,i}\frac{x^i-1}{x-1}\nonumber\\
&=& \frac{1}{x-1} \left(F_{k,\ell -1}(x)-F_{k,\ell -1}(1)\right)\nonumber\\
&=& \Delta F_{k,\ell -1}(x). \label{eq:hkdelta}
\end{eqnarray}
This, together with~\eqref{eq:decomp10}, already allows us to
prove~\eqref{eq:Hk} when $\ell >1$. Indeed, one then has:
$$
F_{k,\ell }(x) = F(x,1)\Delta F_{k,\ell -1}(x) =
(F(x,1)\Delta)^{(\ell) }\left(tx(F(x,1)\Delta)^{(m)}\right)^{(k)} (x),$$
by the induction hypothesis. This is \eqref{eq:Hk} at rank $(k,\ell )$.

It remains to treat the case $\ell =1$. To this end we need to
determine the series 
$F^\circ_{k,0}(x)$. By definition, a pointed interval
$I_1=[P_1^\ell P_1^r,Q_1]$ is non-proper if $P_1^\ell$ is empty, in
which case $I_1$ can 
be identified with the (non-pointed) interval $[P_1^r,Q_1]$.
This implies that
$F^\circ_{k,0}(x)=  F_{k,0}(x)$.
Therefore \eqref{eq:decomp10} and \eqref{eq:hkdelta} give:
\begin{eqnarray*}
  F_{k,1}(x) &=& F(x,1) \left( \Delta F_{k,0}(x) + F_{k,0}(x)\right) \\ 
             &=& F(x,1) \Delta\big(xF_{k,0}(x)\big).
\end{eqnarray*}
 By the induction hypothesis,  $F_{k,0}(x) = \frac{1}{x}
\left(tx(F(x,1)\Delta)^{(m)}\right)^{(k)} (x)$, so that
$$
F_{k,1}(x) = F(x,1)\Delta \left(tx(F(x,1)\Delta)^{(m)}\right)^{(k)}.
$$
We recognise~\eqref{eq:Hk} at rank $(k,\ell =1)$, and this
settles the last case of the induction. 
\end{proof}

\begin{proof}[Proof of Proposition~\rm{\ref{prop:eq}}]
  By Lemmas~\ref{lem:ordinary} and \ref{lemma:H}, and the
  definition~\eqref{htdef} of $\tilde   h_k(y)$,  we have:
  \begin{eqnarray*}
  F(x,y)&=& 
\sum_{k\geq 0}\tilde h_k(y) \left( tx
  (F(x,1)\Delta)^{(m)}\right)^{(k)}(x)=
  \sum_{\lambda} y^{\ell(\lambda)}\frac{p_\lambda}{z_\lambda} \left( tx
  (F(x,1)\Delta)^{(m)}\right)^{(|\lambda|)}(x).
\end{eqnarray*}
Letting $\alpha_i$ be the number of parts equal to $i$ in the partition
 $\lambda$, and summing on the $\alpha_i$'s rather than on $\lambda$, we can
 rewrite this sum as:
 \begin{eqnarray*}
  F(x,y)&=& \sum_{\alpha_1,\alpha_2,\dots}
  \prod_i \left(  \frac{y^{\alpha_i}}{\alpha_i!}\left(\frac{p_i}{i}\right)^{\alpha_i} \left( tx
  (F(x,1)\Delta)^{(m)}\right)^{(i\alpha_i)}\right)(x)\\
  &=&\prod_{i\ge 1} \exp\left(y \,\frac{p_i}{i} \left( tx
  (F(x,1)\Delta)^{(m)}\right)^{(i)}\right) (x)\\
  &=&
  \exp\left(y \sum_i \frac{p_i}{i} \left( tx
  (F(x,1)\Delta)^{(m)}\right)^{(i)}\right) (x).
\end{eqnarray*} 
 We have used the fact that the operator $\Delta$ commutes with the
 multiplication by $y$ and by $p_i$.
This is the second functional equation satisfied by $F(x,y)$ given in
Proposition~\ref{prop:eq}. The third one, \eqref{eq:Fb}, follows by
differentiating with respect to $y$.
\end{proof}

\section{Principle of the proof, and  the case $m=1$}
\label{sec:sol1}
\subsection{Principle  of the proof}
\label{sec:principle}
Let us consider the functional equation~\eqref{eq:Fb}, together with
the initial condition  $F(t,p;x,0)=x$. Perform the change of
variables~\eqref{t-x-param}, and denote 
$G(z,p;u,y)\equiv G(u,y)= F(t,p;x,y)$. Then $G(u,y)$ is a series in $z$
with coefficients in $\GK[u,y]$ (where $\GK=\qs(p_1, p_2, \ldots)$) satisfying 
\beq\label{eq:G}
\frac{\partial G}{\partial y} (u,y)=
\sum_{k\geq 1} \frac{p_k}{k}
\left(z(1+u)e^{-m(K(u)+L)}\left(\frac{uG(u,1)}{(1+u)e^{-mK(u)}-1} 
\ \Delta_u\right)^{(m)} \right)^{(k)}G(u,y),  
\eeq
with
$
\Delta_u H(u)= \frac {H(u)-H(0)}{u},
$
and the initial condition 
\beq\label{init-G}
G(u,0)=(1+u)e^{-mK(u)}.
\eeq
Observe that this pair of equations defines $G(u,y)\equiv G(z,p;u,y)$ uniquely as a
\fps\ in $z$. Indeed, the coefficient of $z^n$ in $G$
can be computed inductively from these equations: one first determines
the coefficient of $z^n$ in $\frac{\partial G}{\partial y}$,
which can be expressed, thanks to~\eqref{eq:G}, in terms of the
coefficients of $z^i$ in  $G$ for $i<n$; then the coefficient of $z
^n$ in $G$ is obtained by integration with respect to $y$, using the
initial condition~\eqref{init-G}.
Hence, if we exhibit
a series $\tilde G(z,p;u,y)$ that satisfies both equations, then
 $\tilde G(z,p;u,y)=G(z,p;u,y)$. We are going to construct such a series.

Let 
\beq\label{tG1}
G_1(z,p;u)\equiv G_1(u)=
({1+\bu}) e^{K(u)+L}\left((1+u)e^{-mK(u)}-1 \right).
\eeq
Then $G_1(u)$ is a series in $z$ with coefficients in $\GK[u]$,
which, as we will see, coincides with $G(u,1)$. Consider now the
following  equation, obtained from~\eqref{eq:G} by replacing
$G(u,1)$ by its conjectured value $G_1(u)$:
\begin{eqnarray}
\frac{\partial \tilde G}{\partial y} (z,p;u,y)&= &
\sum_{k\geq 1} \frac{p_k}{k}
\left(z(1+u)e^{-m(L+K(u))}\left(\frac{uG_1(u)}{(1+u)e^{-mK(u)}-1} 
\ \Delta_u\right)^{(m)} \right)^{(k)} \tG(z,p;u,y) \nonumber
\\
&=&\sum_{k\geq 1} 
\frac{p_k}{k}
\left(z(1+u)e^{-m(L+K(u))}\left(({1+u}) e^{K(u)+L}
\ \Delta_u\right)^{(m)} \right)^{(k)} \tG(z,p;u,y),  \label{eq:Gtilde0}
\end{eqnarray}
with the initial condition 
\beq\label{init:Gtilde}
\tilde G(z,p;u,0)=(1+u)e^{-mK(u)}.
\eeq 
 Eq.~\eqref{eq:Gtilde0}  can be rewritten as 
\beq\label{eq:Gtilde}
\frac{\partial \tilde G}{\partial y} (z,p;u,y)= 
\sum_{k\geq 1} \frac{p_k}{k} \left(zvA(u)^m \Lambda^{(m)}\right)^{(k)}
\tG(z,p;u,y),
\eeq
where 
 \beq\label{A-def}
A(u)=\frac{u}{1+u}e^{-K(u)},
 \eeq
$\Lambda$ is the  operator defined by
\beq\label{Lambda-def}
\Lambda (H) (u) = 
\frac{H(u)-H(0)}{A(u)} ,
\eeq
and  $v={(1+u)^{m+1}}{u^{-m}}$ as before.
Again, it is not hard
to see that~\eqref{eq:Gtilde} and the initial
condition~\eqref{init:Gtilde} define a unique 
series in $z$, denoted $\tG(z,p;u,y)\equiv \tG(u,y)$. The coefficients
of this series lie 
in $\GK[u,y]$.
%
The principle of our proof can be described as follows.

\begin{quote}
{\it  If we prove that $\tG(u,1)=G_1(u)$, then the
  equation~{\rm\eqref{eq:Gtilde0}} satisfied by $\tG$ coincides with the
  equation~\eqref{eq:G} that defines $G$, and thus
  $\tG(u,y)=G(u,y)$. In particular,
  $G_1(z,p;u)=\tG(z,p;u,1)=G(z,p;u,1)=F(t,p;x,1)$, and 
  Theorem~{\rm\ref{thm:main}} is proved. }
\end{quote}

\medskip

\noindent{\bf Remark.} Our proof relies on the fact that we have been
able to \emm guess,  the value of $G(u,1)$, given by~\eqref{tG1}. This
was a difficult task, which we discuss in greater details in
Section~\ref{sec:final-guess}. 

\subsection{The case $m=1$}
\label{subsec:sol1}
Take $m=1$. In this subsection, we describe the three steps that, starting
from~\eqref{eq:Gtilde},  prove that
$\tG(u,1)=G_1(u)$. In passing, we establish the expression~\eqref{F-param-y1} of
$F(t,p;x,1)$ (equivalently, of $\tG(z,p;u,1)$) given in  
Theorem~\ref{thm:1}. 
The case of general $m$ is difficult, and we hope
that studying in details the case $m=1$ will make the ideas of the
proof more transparent. Should this specialization not suffice, we
invite the reader to set further $p_i={\mathbbm 1}_{i=1}$, 
in which
case we are simply counting labelled Tamari intervals (see
also~\cite{mbm-chapuy-preville}). 
%

\subsubsection{A homogeneous differential equation and its solution}
\label{sec:m11}
When $m=1$, the equation~\eqref{eq:Gtilde}  defining $\tG(z;u,y)\equiv
\tG(u,y)$ reads 
\beq\label{ED-uy}
\frac{\partial \tG}{\partial y} (u,y)= \sum_{k\ge 1} \frac{p_k}{k} z^k
\Big((1+u)(1+\bu) \, \Omega \Big)^{(k)}
\tG(u,y),
\eeq
where $\bu=1/u$ and the operator $\Omega $ is defined by $\Omega  H (u) = H(u)-H(0)$,
with the initial condition 
\beq\label{init-1}
\tG(u,0)=({1+u})e^{-K(u)}.
\eeq
These equations imply  that $\tG(-1,y)=0$. 

Observe that the differential equation~\eqref{ED-uy} is not
homogeneous: the term obtained for $k=1$
involves the (unknown) series $\tG(0,y)$, and
more and more unknown series independent of $u$ occur as $k$ grows. By
exploiting the  symmetry of the term $(1+u)(1+\bu)$,
we are going to obtain an equation that does not involve these
series. This idea has already been used for other equations with
divided differences~\cite{mbm-mishna}.
\begin{Lemma}
  \label{lemma:simplereflection}
For  $k\geq 0$ one has:
$$
\Big((1+u)(1+\bu) \Omega \Big)^{(k)}\tG(u,y)
=
(1+u)^k(1+\bu) ^k \tG(u,y) - P_k(v),
$$
where $P_k\in\mathbb{K}[y][ [ z ] ] [ v ]  $ 
and $v=(1+u)(1+\bu)$.
\end{Lemma}
\begin{proof}  This is easily proved by induction on $k$. 
Let us also give  a combinatorial  argument.

Clearly, it
suffices to prove that  
for any $i\ge0$, the above property holds with   $\tG(u,y)$ replaced by
$u^i$.
Consider  walks  on the line $\ns$, starting
from $i$ and consisting of $k$ steps taken in ${-1, 0^1 , 0^2, 1}$ (the
steps $0$ are thus bicolored).  The
term $(1+u)^k(1+\bu) ^k u^i$ in  right-hand side of the above
identity counts these walks by their final position.
The left-hand side counts those walks that  never
reach $0$, except possibly at their final point. 

Hence the difference $P_k$ between the two terms counts walks that reach $0$ before their final point. Such walks consist
of a walk visiting $0$ only at its end, of length, say, $\ell$,
followed by an arbitrary walk of length $k-\ell$. This shows that
$P_k$ has the following form:
$$
P_k= \sum_{\ell=0}^{k-1} a_\ell ((1+u)(1+\bu))^{k-\ell}
$$
where $a_\ell$ is the number of $\ell$-step walks going from $i$ to $0$ and
visiting $0$ only once.
It is now clear that $P_k$ is a polynomial in $v$.
\end{proof}

Observe that the quantity $P_k(v)$, being
a function of $v=(1+u)(1+\bu)$, is left invariant by the substitution $u\mapsto
\bu$.  This symmetry is the keystone of our approach, as it 
enables us to eliminate some {\it a priori} intractable terms in \eqref{ED-uy}.
Replacing $u$ by $\bu$ in~\eqref{ED-uy} gives
$$
\frac{\partial \tG}{\partial y} (\bu,y)= \sum_k \frac{p_k}{k} z^k
\Big((1+u)(1+\bu) \Omega \Big)^{(k)}
\tG(\bu,y),
$$
so that, applying Lemma~\ref{lemma:simplereflection} and using 
$v(u)=v(\bu)$ we obtain:
$$
\frac{\partial}{\partial y} \left( \tG(u,y)-\tG(\bu,y)\right)= 
\sum_{k\geq 1} \frac{p_k}{k}z^k(1+u)^k(1+\bu)^k
\left(\tG(u,y)-\tG(\bu,y)\right)
= V(v) \left(\tG(u,y)-\tG(\bu,y)\right),
$$
where $V(v)$ is given by \eqref{eq:defV}.
This is a \emm homogeneous, linear differential equation satisfied by
$\tG(u,y)-\tG(\bu,y)$. It is readily solved, and  the initial
condition~\eqref{init-1}  yields
\beq\label{tGsol1}
\tG(u,y)-\tG(\bu,y)
=(1+u)\left(e^{-K(u)}-\bu e^{-K(\bu)}\right)e^{yV(v)}.
\eeq

\subsubsection{Reconstruction of $\tG(u,y)$}
Recall that $\tG(u,y)\equiv \tG(z,p;u,y)$ is a series in $z$ with
coefficients in $\GK[u,y]$. Hence, by extracting from the
above equation the positive part in $u$ (as defined in
Section~\ref{sec:notation}), we obtain  
$$
\tG(u,y)-\tG(0,y)=[u^>]\left(
(1+u)\left(e^{-K(u)}-\bu e^{-K(\bu)}\right)e^{yV(v)}
\right).
$$
For any Laurent polynomial $P$, we have
\beq\label{1+u}
[u^>]\left((1+u)P(u)\right)= (1+u)[u^>]P(u)+u [u^0]P(u).
\eeq
Hence
\begin{multline*}
  \tG(u,y)-\tG(0,y)
=(1+u)[u^>]\left(e^{yV(v)}\left(e^{-K(u)}-\bu
e^{-K(\bu)}\right)\right)
\\
+u[u^0]\left(e^{yV(v)}\left(e^{-K(u)}-\bu
e^{-K(\bu)}\right)\right).
\end{multline*}
Since $\tG(-1,y)=0$,  setting $u=-1$ in this equation gives the value of $\tG(0,y)$
(this is an instance of the \emm kernel method,, see \emm e.g.,~\cite{hexacephale}): 
$$
-\tG(0,y)
=-[u^0]\left(e^{yV(v)}\left(e^{-K(u)}-\bu e^{-K(\bu)}\right)\right),
$$
so that finally,
\begin{eqnarray}
\tG(u,y) 
&=&(1+u)[u^>]\left(e^{yV(v)}\left(e^{-K(u)}-\bu 
e^{-K(\bu)}\right)\right) \nonumber
\\
&& \ \hskip 50mm +(1+u)[u^0]\left(e^{yV(v)}\left(e^{-K(u)}-\bu
e^{-K(\bu)}\right) \right) \nonumber
\\
&=&(1+u)[u^\ge]\left(e^{yV(v)}\left(e^{-K(u)}-\bu
e^{-K(\bu)}\right)\right). \label{Gsol-1}
\end{eqnarray}
As explained in
Section~\ref{sec:principle}, $\tG(u,y)=G(u,y)$ will be proved if we
establish that 
$\tG(u,1)$ coincides with the series $G_1(u)$ given by~\eqref{tG1}. This
is the final step of our proof. 
\subsubsection{The case $y=1$}
\label{sec:m=y=1}
 Equation~\eqref{Gsol-1} completely describes the solution
 of~\eqref{ED-uy}. It remains to check that  
$\tG(u,1)=G_1(u)$, that is
\beq\label{G11conj}
\tG(u,1)= ({1+\bu})e^{K(u)+L}\left( (1+u)e^{-K(u)}-1\right).
\eeq
Let us set $y=1$ in~\eqref{Gsol-1}. We find, using  $V(v)=K(\bu)+L+K(u)$:
\begin{eqnarray*}
  \tG(u,1)
&=&  (1+u)[u^\ge]\left(e^{L+K(\bu)}-\bu e^{L+K(u)}\right)\\
&=& (1+u) e^L \left(1- \bu e^{K(u)}+\bu \right),
\end{eqnarray*}
which coincides with~\eqref{G11conj}. Hence
$\tG(z,p;u,y)=G(z,p;u,y)=F(t,p;x,y)$ (with the change 
of variables~\eqref{t-x-param}), and Theorem~\ref{thm:1} is proved, using~\eqref{Gsol-1}.

\section{Solution of the functional equation: the general case}
\label{sec:sol}
We now adapt to the general case the solution described for $m=1$ in
Section~\ref{subsec:sol1}.
Recall from Section~\ref{sec:principle} that we start
from~\eqref{eq:Gtilde}, and want to prove that $\tG(u,1)=G_1(u)$.  We
first obtain in Section~\ref{sec:homogeneous} the counterpart of~\eqref{tGsol1}, that is,
an explicit  expression for a linear combination of the series
$\tG(u_i,y)$, where  $u_0=u, u_1, \ldots, 
u_m$ are the $m+1$ roots of the equation $v(u)=v(u_i)$, with
$v(u)=(1+u)^{m+1}\bu^m$.  In Section~\ref{sec:reconstruct}, we  reconstruct from this
expression the series  $\tG(u,y)$, by taking iterated positive
parts. This generalizes~\eqref{Gsol-1}. The third
part of the proof differs from Section~\ref{sec:m=y=1}, because we are not able
to derive from our expression of
$\tG(u,y)$ that $\tG(u,1)=G_1(u)$. Instead, the arguments of
Section~\ref{sec:reconstruct} imply  that the counterpart
of~\eqref{tGsol1}  also has a unique solution when
$y=1$, and we  check that $G_1(u)$ is a solution. 

\subsection{A homogeneous differential equation and its solution}
\label{sec:homogeneous}
\label{sec:lemma-implicit-functions}

Let us return to the equation \eqref{eq:Gtilde} satisfied by $\tilde G(u,y)$.
This equations involves the quantity
$$
v\equiv v(u)=(1+u)^{m+1} \bu^m 
$$
 In the case $m=1$, this (Laurent) polynomial was
$(1+u)(1+\bu)$, 
and took the same value for $u$ and $\bu$. 
We are again interested in the series $u_i$ such that $v(u_i)=v(u)$.
\begin{Lemma}\label{lem:ui}
  Denote $v
=(1+u)^{m+1} u^{-m}$, and  consider the
  following polynomial equation in $U$:
$$
(1+U)^{m+1}=U^m v .
$$
This equation has no double root. We denote its $m+1$ roots by $ u_0=u,
u_1, \dots, u_m$.
\end{Lemma}
\begin{proof}
A double root  would also satisfy
$$
(m+1) (1+U)^m =mU^{m-1}v,
$$
and this is easily shown to be impossible.
\end{proof}

\noindent{\bf Remark.} One can 
 express the $u_i$'s as
Puiseux series in $u$ (see~\cite[Ch.~6]{stanley-vol2}), but this will
not be needed here, and we will think of them as abstract elements of an
 algebraic extension of $\qs(u)$. In fact, in this paper, the $u_i$'s
 always occur 
in \emm symmetric rational functions, of the $u_i$'s, which are thus  rational functions of $v$. At
some point, we will have to prove that a symmetric polynomial in the
$u_i$'s (and thus a polynomial in $v$) vanishes at $v=0$, that is, at
$u=-1$, and we will then  consider series expansions of the $u_i$'s around $u=-1$.

\medskip
The following proposition generalizes~\eqref{tGsol1}.
\begin{Proposition}\label{prop:combi-lin-m}
 Denote $v=(1+u)^{m+1}
u^{-m}$, and let the series $u_i$ be defined as above. Denote
$A_i=A(u_i)$, where $A(u)$ is given by~\eqref{A-def}. Then 
\begin{eqnarray}\label{eq:combi-lin-y}
\sum_{i=0}^m \frac{\tilde G(u_i,y)}{\prod_{j\neq i} (A_i-A_j)} = v e^{yV(v)}.
\end{eqnarray}
\end{Proposition}
\noindent By $\prod_{j\neq i} (A_i-A_j)$ we mean $\prod_{0\le j \le m, j\neq i}
(A_i-A_j)$ but we prefer the shorter notation when the bounds
on $j$ are clear.
 Observe that the $A_i$'s are distinct since the $u_i$'s are distinct
(the coefficient of $z^0$ in $A(u)$ is $u/(1+u)$). Note also that
when $m=1$, then $u_0=u$, $u_1= \bu$, and~\eqref{eq:combi-lin-y}
coincides with~\eqref{tGsol1}. 
In order to prove the proposition, we need the following  two lemmas.
\begin{Lemma}\label{lemma:Lagrange}
Let $x_0, x_1, \dots , x_m$ be $m+1$  variables. Then
\beq\label{eq:Lagrange-inv}
\sum_{i=0}^m \frac{{x_i}^m}{\prod_{j\neq i}{(x_i-x_j)}} = 1
\eeq
and
\beq\label{eq:Lagrange-inv-bis}
\sum_{i=0}^m \frac{1/x_i}{\prod_{j\neq i}{(x_i-x_j)}} = (-1)^{m} \prod_{i=0}^m
\frac{1}{x_i}.
\eeq

\smallskip \noindent
Moreover, for any polynomial $Q$ of degree less than $m$,
\begin{eqnarray}\label{eq:Lagrange-poly}
\sum_{i=0}^m \frac{Q(x_i)}{\prod_{j\neq i}{(x_i-x_j)}} = 0.
\end{eqnarray}
\end{Lemma}
\begin{proof}
By Lagrange interpolation, any polynomial $R$ of degree at most $m$
satisfies:
$$
R(X)=\displaystyle\sum_{i=0}^m {R(x_i)}
\prod_{j\neq i}\frac{X-x_j}{x_i-x_j}.
$$
Equations \eqref{eq:Lagrange-inv-bis} and~\eqref{eq:Lagrange-poly} follow  by
evaluating this equation at $X=0$, respectively with $R(X)=1$
and $R(X)=X Q(X)$.
Equation~\eqref{eq:Lagrange-inv} is obtained by taking $R(X)=X^m$ and
extracting the leading coefficient. 
\end{proof}

Our second lemma replaces Lemma~\ref{lemma:simplereflection} for general values of $m$.
\begin{Lemma}\label{lemma:implicit-functions} 
Let $k\geq 0$, and let $\Lambda$ be the operator defined by~\eqref{Lambda-def}. 
Let  $H(z;u)\equiv H(u)$ be a \fps\ in $z$,  having coefficients in
$\GL(u)$, with  $\GL=\GK(y)$.
Assume that these coefficients have no pole at $u=0$.
Then there exists a polynomial $P_k(X,Y) \in \GL [[ z ]] [X,Y]$ of
degree less than $m$ in $X$, such that
\beq\label{eq:operateuriter-A}
\Big(z v A(u)^m\Lambda^{(m)} \Big)^{(k)} H(u) = (zv)^k 
H(u) - P_k(A(u),v).
\eeq
\end{Lemma}
\begin{proof}
We give here a power series argument, but an analogue of the combinatorial
argument proving  Lemma~\ref{lemma:simplereflection} (carefully
justified) also exists.

We denote by $\mathcal{L}$ the subring of $\GL(u)[[z]] $
  formed by formal power series whose coefficients have no pole at $u=0$. 
  By assumption, $H(u)\in \mathcal{L}$.
  We  use the notation $O(u^k)$   to denote an element of $\GL(u)[
    [z]]$ of the   form $u^k J(z;u)$ with $J(z;u)\in \mathcal{L}$.

  First, note that  $A(u)=ue^{-K(u)}/(1+u)$
belongs to $\cL$.
Moreover, 
  \beq\label{eq:Adev}
   A(u)=u+ O(u^2).
   \eeq  
  We 
 first prove 
  that for all series $I(u)\in \mathcal{L}$, there exists a sequence of formal power series
  $(g^I_j )_{j\geq0}\in\GL[[z]]^\ns$ such that for all $\ell\geq 0$, 
  \begin{eqnarray}\label{eq:implicit-functions}
    I(u) = \sum_{j=0}^{\ell-1} g^I_j  A(u)^j + O(u^{\ell}).
  \end{eqnarray}
We prove~\eqref{eq:implicit-functions} by
induction on $\ell \ge 0$.   The identity holds for $\ell=0$ since
$I(u)\in\mathcal{L}$. Assume it holds for some $\ell\geq
  0$:  there exists 
series $g^I_0, \ldots, g^I_{\ell-1}$ in $\GL[[z]]$
  and $J(u)\in\mathcal{L}$ such that
  \begin{eqnarray*}
    I(u) = \sum_{j=0}^{\ell-1} g^I_j  A(u)^j + u^{\ell} J(u).
  \end{eqnarray*}
  By~\eqref{eq:Adev} and by induction on $r$, we have 
  $u^r = A(u)^r +O(u^{r+1})$ for all $r\geq 0$. 
Using this identity with $r=\ell$, and rewriting $J(u)=J(0)+O(u)$,
we obtain $u^\ell   J(u) = J(0) A(u)^\ell +O(u^{\ell+1})$,
  so  that:
  \begin{eqnarray*}
    I(u) = \sum_{j=0}^{\ell} g^I_j  A(u)^j + O(u^{\ell+1}),
  \end{eqnarray*}
  with $g^I_\ell :=J(0) \in \GL[[z]]$. Thus~\eqref{eq:implicit-functions} holds for $\ell+1$.

\smallskip
  We now prove that for all $q\geq 0$, one has:
 \beq\label{eq:Lambdaiter-A}
\Lambda^{(q)} I(u)= \frac{1}{A(u)^q} 
\left(I(u) - \sum_{j=0}^{q-1} g^I_j A(u)^{j}\right),
\eeq 
where the series $g^I_j $ 
are those that satisfy~\eqref{eq:implicit-functions}.
Again, we proceed by
  induction on $q\ge 0$. The identity  clearly holds for $q=0$. Assume
  it holds for some $q\geq 0$. 
In~\eqref{eq:Lambdaiter-A}, replace
  $I(u)$ by its expression~\eqref{eq:implicit-functions} obtained with
  $\ell=q+1$, and let $u$ tend to $0$: this  shows that
  $g^I_q $ is in fact $\Lambda^{(q)}I(0)$. From the definition of
  $\Lambda$ one then   obtains 
$$
\Lambda^{(q+1)}I
  (u)=\frac{\Lambda^{(q)}I(u)-g^I_q }{A(u)}=\frac{1}{A(u)^{q+1}}\left(I(u)
    - \sum_{j=0}^{q} g^I_j A(u)^{j}\right).
$$
  Thus~\eqref{eq:Lambdaiter-A} holds for $q+1$.

\smallskip
We  finally prove, by induction on $k\geq 0$, that
\eqref{eq:operateuriter-A} holds and that the left-hand side of
\eqref{eq:operateuriter-A} is an element of $\mathcal{L}$.
For $k=0$, these results are clear, with $P_0=0$.  Assume
they hold for some $k\ge 0$, for any $H(u)\in\mathcal{L}$.
Let $H(u)\in \mathcal{L}$ and let $M(u)$ 
be the left-hand side of~\eqref{eq:operateuriter-A}. 
 By the induction hypothesis, $M(u)\in\mathcal{L}$, so that
applying 
\eqref{eq:Lambdaiter-A} with $I(u)=M(u)$ and $q=m$ gives:
\begin{eqnarray}
  zvA(u)^m\Lambda^{(m)} M(u)
  &=&zv \left(M(u) - \sum_{j=0}^{m-1} g^M_j A(u)^y\right).\label{eq:devk+1}
\\
&=&\left(zvA(u)^m\Lambda^{(m)}\right)^{(k+1)}H(u) \hskip 20mm \hbox{by definition of } M.
  \nonumber 
\end{eqnarray}
By the induction hypothesis~\eqref{eq:operateuriter-A}, we have $M(u) = (zv)^k
H(u) - P_k(A(u),v)$ with $P_k(X,Y)$ of degree less than $m$ in $X$, so
that the above equation gives:
\begin{eqnarray*}
  \left(zvA(u)^m\Lambda^{(m)}\right)^{(k+1)}H(u) 
  &=& (zv)^{k+1} H(u) - P_{k+1}(A(u),v),
\end{eqnarray*}
with 
$$
P_{k+1}(X,Y):= zY \left(P_k(X,Y)+ \sum_{j=0}^{m-1} g^M_j  X^j
\right).
$$
 Note that $P_{k+1}(X,Y)$ still has degree 
 less than $m$ in $X$. 

It remains to prove that
$\left((zvA(u)^m\Lambda^{(m)}\right)^{(k+1)}H(u)\in\mathcal{L}$. 
Applying \eqref{eq:implicit-functions} with $I(u)=M(u)$ and $\ell=m+1$, 
and substituting in \eqref{eq:devk+1}, we obtain:
\begin{eqnarray*}
  \left(zvA(u)^m\Lambda^{(m)}\right)^{(k+1)}H(u)&=& zv \left(g^M_m  A(u)^m +O(u^{m+1})\right) \\
  &=& z v  u^m \left(g^M_m +O(u)\right), 
\end{eqnarray*}
since $A(u)^m = u^m + O(u^{m+1})$.
Since $v=(1+u)^{m+1}u^{-m}$, this shows that  
$\left(zvA(u)^m\Lambda^{(m)}\right)^{(k+1)}H(u)$ belongs to $\mathcal{L}$, which 
completes the proof.
\end{proof}

\begin{proof}[Proof of Proposition~{\rm\ref{prop:combi-lin-m}}]
  Thanks to Lemma~\ref{lemma:implicit-functions}, we can rewrite
  \eqref{eq:Gtilde} as 
\begin{eqnarray}\label{eq:Gy-poly}
\frac{\partial \tilde G}{\partial y} (u,y)=
\sum_{k\geq 1}
\frac{p_k}{k} \Big((z v)^k \tilde G(u,y) - P_k(A(u),v)\Big),
\end{eqnarray}
where $v\equiv v(u)=(1+u)^{m+1}\bu^m$, and for all $k\geq 1$, $P_k(X,Y)$ is a polynomial of
degree less than $m$ in $X$ with coefficients in 
$\mathbb{K}(y)[[ z ]] $.

As was done in Section~\ref{sec:m11}, we are going to use the fact that
$v(u_i)=v$ for all  $i\in\llbracket 0, m\rrbracket$ to eliminate the
(infinitely many) unknown polynomials $P_k(A(u),v)$. For $0\le i \le m$,  the
substitution $u\mapsto u_i$ in~\eqref{eq:Gy-poly} gives:
\begin{eqnarray}\label{eq:Gy-poly-ui}
\frac{\partial \tilde G}{\partial y} (u_i,y)=
\sum_{k\geq 1}
\frac{p_k}{k} \Big( (zv)^k \tilde G(u_i,y) - P_k(A_i,v) \Big),
\end{eqnarray}
with $A_i=A(u_i)$. Consider the linear combination
\beq\label{combine}
R(u,y):=
\sum_{i=0}^m \frac{\tilde G(u_i,y)}{\prod_{j\neq i}{(A_i-A_j)}}.
\eeq
Recall that $A_i$ is independent of $y$. 
Thus   by~\eqref{eq:Gy-poly-ui},
$$
\begin{array}{llll}
\displaystyle \frac{\partial R}{\partial y} (u,y)&=&
\displaystyle \sum_{k\geq 1}\frac{p_k}{k}\left( 
 (z v)^k R(u,y) - \sum_{i=0}^m\frac{P_k(A_i,v)}{\prod_{j\neq i}{(A_i-A_j)}} \right),
\\
&=&
\displaystyle \sum_{k \geq 1}\frac{p_k}{k} (z v)^k R(u,y), &  \hskip 10mm \hbox{ by~\eqref{eq:Lagrange-poly}},
\\
&=& V(v) R(u,y),
\end{array}
$$
where $V(v)$ is defined by~\eqref{eq:defV}. This homogeneous linear differential equation is readily solved:
$$
R(u,y)=R(u,0) e^{yV(v)}.
$$
Recall the expression~\eqref{combine} of $R$ in terms of $\tG$.
The initial 
condition~\eqref{init:Gtilde} can be rewritten $\tG(u,0)=v A(u)^m$, which yields
$$
\begin{array}{llll}
R(u,0)& =& \displaystyle v\sum_{i=0}^m \frac{ {A_i}^m}{\prod_{j\neq i}{(A_i-A_j)}} 
\\
&=& v & 
\end{array}
$$
by~\eqref{eq:Lagrange-inv}. Hence $R(u,y)=ve^{yV(v)}$, and the proposition is proved.
\end{proof}

\subsection{Reconstruction of $\tG(u,y)$}
\label{sec:reconstruct}
We are now going to prove that~\eqref{eq:combi-lin-y}, together with the
condition $\tG(-1,y)=0$ 
derived from~(\ref{eq:Gtilde0}--\ref{init:Gtilde}),
 characterizes the series $\tG(u,y)$. We will actually obtain 
a (complicated) expression for
this series, generalizing~\eqref{Gsol-1}. 

We first introduce some notation.  Consider a \fps\ in $z$, denoted $H(z;u)\equiv
H(u)$, having coefficients in $\GL[u]$,
where  $\GL=\GK(y)$.
We define a series $H_k$ in $z$ whose coefficients are rational symmetric
functions  of $k+1$ variables $x_0, \ldots, x_k$: 
\beq\label{def-Hk}
H_k(x_0, \ldots, x_k)= \sum_{i=0}^k \frac{H(x_i)}{\displaystyle \prod_{0\le j \le k,
    j\not = i} (A(x_i)-A(x_j))},
\eeq
where, as above, $A$ is defined by~\eqref{A-def}.
\begin{Lemma}\label{lem:Hk}
  The series $H_k(x_0, \ldots, x_k)$ has coefficients in $\GL[x_0,
    \ldots , x_k]$. If, moreover, $H(-1)=0$,
then
the coefficients of $H_k$ are multiples of $(1+x_0) \cdots (1+x_k)$.
\end{Lemma}
\begin{proof}
Using the fact that $e^{-K(u)}=1+O(z)$,
 it is not hard to prove that
\beq\label{AB}
\frac 1 {A(x_i)-A(x_j)}= \frac 1{x_i-x_j} B(x_i,x_j),
\eeq
where $B(x_i,x_j)$ is a series in $z$ with \emm polynomial,
coefficients in $x_i$ and $x_j$.
Hence 
$$
 H_k(x_0, \ldots, x_k) \prod_{0\le i<j\le k} (x_i-x_j)
$$
has  polynomial coefficients in the $x_i$'s. But these polynomials are
anti-symmetric in the $x_i$'s (since $H_k$ is symmetric), hence they must be
multiples of the Vandermonde $\prod_{i<j} (x_i-x_j)$. Hence $H_k(x_0,
\ldots, x_k)$ has polynomial coefficients.

\medskip
A stronger property than~\eqref{AB} actually holds, namely:
$$
\frac 1 {A(x_i)-A(x_j)}= \frac {(1+x_i)(1+x_j)}{x_i-x_j}\, C(x_i,x_j),
$$
where $C(x_i,x_j)$ is a series in $z$ with  polynomial
coefficients in $x_i$ and $x_j$. Hence, 
if $H(-1)=0$, that is, 
if $H(x)=(1+x)K(x)$ where $K(x)$ has polynomial coefficients in $x$,
$$
H_k(x_0, \ldots, x_k)= \sum_{i=0}^k 
{K(x_i)}(1+x_i)^{k+1}\prod_{j\not = i}\frac{(1+x_j)C(x_i,x_j)}{ x_i-x_j}.
$$
Setting $x_0=-1$ shows that $H_k(-1,x_1, \ldots, x_k)=0$, so that
$H_k(x_0, \ldots, x_k)$ is a multiple of $(1+x_0)$. By symmetry, it is
also a multiple of all $(1+x_i)$, for $1\le i \le k$.
\end{proof}

Our treatment of~\eqref{eq:combi-lin-y} actually applies to equations
with an arbitrary right-hand side. We consider a  \fps\   $H(z;u)\equiv
H(u)$ with coefficients in $\GL[u]$,
  satisfying $H(-1)=0$ and 
$$
\sum_{i=0}^m \frac{H(u_i)}{\prod_{j\neq i} (A_i-A_j)} = \Phi_m(v),
$$
for some series $\Phi_m(v)\equiv\Phi_m(z;v)$ with coefficients in
$v\GL[v]$, where $v=(1+u)^{m+1}\bu^m$.
Using the  notation~\eqref{def-Hk}, this equation can be rewritten as
$$
H_m(u_0, \ldots, u_m)= \Phi_m(v) .
$$
We will give an explicit expression of $H(u)$ involving two standard
families of symmetric functions,
namely the homogeneous functions $h_\lambda$ 
and the monomial functions $m_\lambda$.

\begin{quote}
{\noindent \bf Caveat. }
{\it These symmetric functions will be evaluated at $(u_0,u_1,\dots
  u_m)$ or $(A(u_0), \ldots, A(u_m))$. They have nothing to do with the
  variables $p_k$ involved in the \gf\ $F^{(m)}(t,p;x,y)$.}
\end{quote}
We also use the following notation:  For any subset
$V=\{i_1, \ldots, i_k\}$ of $\llbracket 0,  m\rrbracket$, of
cardinality $k$,  and any sequence $(x_0, \ldots, x_m)$, we
denote $x_V=\{ x_{i_1}, \ldots, x_{i_k}\}$.

\begin{Proposition}\label{prop:extraction}
  Let $H(z;u)\equiv H(u)$ be a power series in $z$ with coefficients
  in $\GL[u]$,   satisfying $H(-1)=0$ and
\beq\label{eq-sym}
H_m(u_0, \ldots, u_m)= \Phi_m(v) ,
\eeq
where $\Phi_m(v)\equiv\Phi_m(z;v)$ is a series in $z$ with coefficients in
$v\GL[v]$. 

There exists a 
 sequence $\Phi_0, \ldots, \Phi_m$ of series in $z$ with
coefficients in $v\GL[v]$, which depend only on $\Phi_m$,  such that
for $0\le k \le m$, and for all
subset $V$ of $\llbracket 0,m\rrbracket$ of cardinality $k+1$,
\beq\label{eq-sym-k}
H_k(u_V
)= \sum_{j=k}^m \Phi_j(v)
h_{j-k}(A_V
).
\eeq
In particular, $H(u)\equiv H_0(u)$ is completely determined 
if $\Phi_m$ is known:
$$
H(u)= \sum_{j=0}^m \Phi_j(v) A(u)^j.
$$
The series   $\Phi_k(v)\equiv \Phi_k(z;v)$  can be computed by a
descending induction on $k$ as follows. Let us denote by
$\Phi_{k-1}^>(u)$ the positive part in $u$ of $\Phi_{k-1}(v)$, that is
$$
\Phi_{k-1}^>(u):=  [u^{> }]\Phi_{k-1}(\bu^m(1+u)^{m+1}).
$$
Then for $1\le k\le m$, this series can be expressed in terms of
$\Phi_k, \ldots, \Phi_m$:
\begin{eqnarray}
\Phi_{k-1}^>(u)&=&
-\frac 1{{m\choose k} }[u^{>}] \left(\sum_{j=k}^m \Phi_j(v) \sum_{\lambda \vdash j-k+1} {m-\ell(\lambda)
  \choose k -\ell(\lambda)} m_\lambda(A_1, \ldots, A_m)\right).\label{phi-rec}
\end{eqnarray}
The extraction makes sense since, as will be seen, $vm_\lambda(A_1,
\ldots, A_m)$  belongs to $\GK[u,\bu][[z]]$.
Finally, $\Phi_{k-1}(v)$ can be expressed in terms of $\Phi_{k-1}^>$: 
\beq\label{Phi-k}
\Phi_{k-1}(v)= \sum_{i=0}^m \left(  \Phi_{k-1}^>(u_i)-  \Phi_{k-1}^>(-1)\right).
\eeq
\end{Proposition}

We first establish three lemmas dealing with  symmetric functions of
the series $u_i$ defined  in Lemma~\ref{lem:ui}.
\begin{Lemma}\label{lem:elem}
  The elementary symmetric functions of $u_0=u, u_1, \ldots, u_m$ are 
$$
e_{j}(u_0,u_1, \ldots, u_m)= 
(-1)^{j}{m+1\choose j} + v{\mathbbm 1}_{j=1}
$$
with $v=u^{-m}(1+u)^m$.

The elementary symmetric functions of $u_1, \ldots, u_m$ are
$$
e_{m-j}(u_1, \ldots, u_m)=
\left\{
\begin{array}{ll}
1 & \hbox{if } j=m,
\\
(-1)^{m-j-1}\sum_{p=0}^j{m+1 \choose p} u^{p-j-1}
& \hbox{otherwise}.
\end{array}
\right.
$$
In particular, they are polynomials in $1/u$, and so is any symmetric
polynomial in $u_1, \ldots, u_m$.

Finally,
$$
\prod_{i=0}^m(1+u_i)=v.
$$
\end{Lemma}
\begin{proof}
The symmetric functions of the roots of a polynomial can be read from
the coefficients of this polynomial. Hence the first result follows
directly from the equation satisfied by the  
$u_i$'s, for $0\le i\le m$, namely  
$$
(1+u_i)^{m+1}=v u_i^m.
$$
 For the second one, we need to find the equation satisfied by $u_1, \ldots,
u_m$, which is 
  $$
0= \frac{(1+u_i)^{m+1}u^m -(1+u)^{m+1}u_i^m}{u_i-u}=
 u^m u_i^m -\sum_{j=0}^{m-1} u_i^j u^{m-j-1} \sum_{p=0}^j{m+1 \choose p} u^{p}.
$$
The second result follows.

The third one is obtained by evaluating at $U=-1$
the identity
$$
\prod_{i=0}^{m}(U-u_i)=(1+U)^{m+1}-v U^m.
$$
\end{proof}
\begin{Lemma}\label{P-reconstruct}
  Denote $v=\bu^m(1+u)^{m+1}$. Let $P$ be a polynomial. Then $P(v)$ is a
  Laurent polynomial in $u$. Let $P^>(u):=[u^>]P(v)$ denote its positive part.
Then
\beq\label{P-dev}
P(v)= P(0)+\sum_{i=0}^m (P^>(u_i)-P^>(-1)).
\eeq
\end{Lemma}
\begin{proof}
The right-hand side of~\eqref{P-dev} is a symmetric polynomial of
$u_0, \ldots, u_m$, and thus, by the first part of
Lemma~\ref{lem:elem}, a polynomial in $v$. Denote it by $\tilde
P(v)$. The second part of
Lemma~\ref{lem:elem} implies that the positive part of $\tilde
P(v)$ in $u$ is
$P^>(u_0)=P^>(u)$. That is, $P(v)$ and $\tilde P(v)$ have the same 
positive part in $u$. In other words, the polynomial
$Q:=P-\tilde P$ is such that $Q(v)$ is a Laurent polynomial in $u$ of
non-positive degree. But since $v=(1+u)^{m+1} \bu^m$, the degree in
$u$ of $Q(v)$ 
coincides with the degree of $Q$, and so $Q$  must be a
constant. Finally, by setting $u=-1$ in $\tilde P(v)$, we see that
$\tilde P(0)= P(0)$ (because  $u_i=-1$ for all $i$ when
$u=-1$, as follows for instance from Lemma~\ref{lem:elem}). Hence
$Q=0$ and the lemma is proved. 
\end{proof}
\begin{Lemma}\label{lem:sym}
  Let $0\le k\le m$, and let $R(x_0, \ldots, x_k)$ be a symmetric
  rational function   of $k+1$ variables $x_0, \ldots, x_k$, such that
  for any  
subset $V$ of $\llbracket 0, k\rrbracket $ of cardinality $k+1$,
$$
R(u_V
)=R(u_0, \ldots, u_k) .
$$
Then there exists a rational fraction in $v$ equal
to  $R(u_0, \ldots, u_k)$.
\end{Lemma} 
\begin{proof} 
  Let $\tilde R$ be the following rational function in $x_0, \ldots, x_m$:
$$
\tilde R (x_0, \ldots, x_m)= 
\frac 1 {{{m+1 }\choose {k+1}}}\sum_{V \subset \llbracket 0, m\rrbracket, \
  |V|=k+1} R(x_V).
$$ 
Then $\tilde R$ is a symmetric function of $x_0, \ldots, x_m$, and
hence a rational function in the elementary symmetric functions
$e_j(x_0, \ldots, x_m)$, say $S(e_1(x_0, \ldots, x_m),  \ldots, e_{m+1}(x_0, \ldots, x_m))$. 
By assumption, 
$$
\tilde R (u_0, \ldots, u_m)=S(e_1(u_0, \ldots, u_m),  \ldots,
e_{m+1}(u_0, \ldots, u_m))= R(u_0, \ldots, u_k). 
$$
  Since $S$ is a rational function, it follows from the first part of
  Lemma~\ref{lem:elem} that $ R(u_0, 
  \ldots, u_k)$ can be written as a rational function in $v$.
\end{proof}

\begin{proof}[Proof of Proposition~{\rm\ref{prop:extraction}}]
  We prove~\eqref{eq-sym-k} by descending induction on $k$. For
  $k=m$, it holds by assumption.
Let us assume that~\eqref{eq-sym-k}  holds for some $k>0$, and
  prove it for $k-1$.

Observe that
$$
(A(x_{k-1})-A(x_k))H_k(x_0, \ldots, x_k)= H_{k-1}(x_0, \ldots, x_{k-2}, x_{k-1})-
H_{k-1}(x_0, \ldots, x_{k-2}, x_{k}).
$$
This is easily proved by collecting the coefficient of $H(x_i)$, for
all $i\in\llbracket 0, k\rrbracket$, in both sides of the equation. We also have, for any
indeterminates $a_0, \ldots, a_m$,
$$
(a_{k-1}-a_k)h_{j-k}(a_0, \ldots, a_k)= h_{j-k+1}(a_0, \ldots,
a_{k-2},a_{k-1})-h_{j-k+1}(a_0, \ldots, a_{k-2},a_{k}).
$$
Let $V$ be a subset of $\llbracket 0,m\rrbracket$ of cardinality
$k-1$, and let $p$ and $q$ be two elements of $\llbracket
0,m\rrbracket\setminus V$. Multiplying~\eqref{eq-sym-k} by 
$A_p-A_q$, 
and using the two equations above 
gives
$$
  H_{k-1}(
u_V,u_p)
-\sum_{j=k}^m \Phi_j(v) h_{j-k+1}(
A_V,A_p)
=\\
H_{k-1}(
u_V,u_q)
-\sum_{j=k}^m \Phi_j(v) h_{j-k+1}(
A_V,A_q).
$$ 
This implies that the series
$$
H_{k-1}(x_0, \ldots,  x_{k-1})-\sum_{j=k}^m \Phi_j(v)
h_{j-k+1}(A(x_0), \ldots,A(x_{k-1})) 
$$
takes the same value at all points 
$u_V$, for $V\subset \llbracket 0,m\rrbracket$ of cardinality $k$.
Hence  Lemma~\ref{lem:sym}, applied to the coefficients of this
series, implies that there exists a series in $z$ with \emm rational, 
coefficients in $v$, denoted $\Phi_{k-1}(v)$, such that for all
$V \subset \llbracket 0,m\rrbracket$ with $|V|=k$:
\beq\label{eq-sym-k-1}
H_{k-1}(
u_V)
-\sum_{j=k}^m \Phi_j(v) h_{j-k+1}(
A_V)= \Phi_{k-1}(v).
\eeq
This is exactly~\eqref{eq-sym-k} with $k$ replaced by $k-1$.

\medskip
The next point we will prove is that the coefficients of
$\Phi_{k-1}$ belong to $v \GL[v]$. In order to do so, we
symmetrize~\eqref{eq-sym-k-1} over $u_0, \ldots, u_m$.
By~\eqref{eq-sym-k-1},
\beq\label{e-neg:a}
{m+1\choose k} \Phi_{k-1}(v)
=\sum_{V\subset \llbracket 0, m\rrbracket, |V|=k} H_{k-1}(u_V)-\sum_{j=k}^m
\left(\Phi_j(v) \sum_{V\subset \llbracket 0, m\rrbracket, |V|=k}h_{j-k+1}(A_V) \right).
\eeq
We will prove that both sums in the right-hand side of this equation
are series in $z$ with coefficients in $v\GL[v]$.

By Lemma~\ref{lem:Hk},
$$
\sum_{V\subset \llbracket 0, m\rrbracket, |V|=k} H_{k-1}(x_V)
$$
is a series in $z$ with polynomial
coefficients in $x_0, \ldots, x_m$, which is symmetric in these
variables. By Lemma~\ref{lem:elem}, the first sum
in~\eqref{e-neg:a} is thus 
a series in $z$ with \emm polynomial, coefficients in $v$. We still
need to prove that this series  vanishes at $v=0$, that is, at
$u=-1$. But this follows from the second part of Lemma~\ref{lem:Hk},
since $u_i=-1$ for all $i$ when $u=-1$.

 Let us now consider the second sum in~\eqref{e-neg:a}, and more
specifically the term
\beq\label{second-sum:a}
\Phi_j(v)  \sum_{V\subset  \llbracket 0,  m\rrbracket, |V|=k}h_{j-k+1}(A_V) .
\eeq
Recall that $$
A_i= \frac{u_i}{1+u_i}\, {e^{-K(u_i)}}.
$$
But by Lemma~\ref{lem:elem},
$$
\frac 1 {1+u_i}= \frac 1 v \prod_{0\le j \not = i \le m} (1+u_j).
$$
Hence~\eqref{second-sum:a} can be written as a series in $z$ with
coefficients in $\GL[1/v, u_0, \ldots, u_m]$, symmetric in
$u_0, \ldots, u_m$. By the first part of Lemma~\ref{lem:elem}, these
coefficients belong to 
$\GL[v, 1/v]$. We want to prove that they actually
belong to $v\GL[v]$, that is, that they are not singular at
$v=0$ (equivalently, at $u=-1$) and even vanish at this point.
From the equation $(1+u_i)^{m+1}= vu_i^m$, it follows that
we can label $u_1, \ldots, u_m$ in such a way that
$$
1+u_i= \xi^i(1+u)+o(1+u),
$$
where $\xi$ is a primitive $(m+1)^{\hbox{\small st}}$ root of unity. Since
$\Phi_j(v)$ is a multiple of $v=\bu^m(1+u)^{m+1}$, and the
symmetric function $h_{j-k+1}$ has degree $j-k+1 \le m$, it follows
that the series~\eqref{second-sum:a} is not singular at $u=-1$, and even
vanishes at this point. Hence its coefficients  belong to $v\GL[v]$.

\medskip
So far, $\Phi_{k-1}(v)$ has been expressed in terms of $H$ (and the
series $\Phi_j$), and we 
now want to obtain an  expression in terms of the $\Phi_j$ only. 
Lemma~\ref{P-reconstruct}, together with $\Phi_{k-1}(0)=0$,
establishes~\eqref{Phi-k}. To express $\Phi^>_{k-1}(u)$, we now
symmetrize~\eqref{eq-sym-k-1} over $u_1, \ldots, u_m$. With the
above notation, 
\beq\label{e-neg}
{m\choose k} \Phi_{k-1}(v)=
\sum_{V\subset  \llbracket 1,  m \rrbracket , |V|=k} H_{k-1}(u_V)-\sum_{j=k}^m
\left(\Phi_j(v) \sum_{V\subset  \llbracket 1, m\rrbracket, |V|=k}h_{j-k+1}(A_V)
\right).
\eeq
As above,
$$
\sum_{V\subset \llbracket 1, m\rrbracket, |V|=k} H_{k-1}(x_V)
$$
is a series in $z$ with polynomial
coefficients in $x_1, \ldots, x_m$, which is symmetric in these
variables. By the second part of Lemma~\ref{lem:elem}, the first sum in~\eqref{e-neg} is thus
a series in $z$ with polynomial coefficients in $1/u$. Since
$\Phi_{k-1}(v)$ has coefficients in $\GL[v]$, and hence in
$\GL[u,1/u]$, the second sum in~\eqref{e-neg} is also a zeries in
$z$ with coefficients in $\GL[u,1/u]$. We can now extract from~\eqref{e-neg} 
 the positive part in $u$, and this gives
$$
{m\choose k} \Phi^>_{k-1}(u)= -[u^{>}]\left(\sum_{j=k}^m
\left(\Phi_j(v) \sum_{V\subset \llbracket 1,  m\rrbracket, |V|=k}h_{j-k+1}(A_V)
\right)\right).
$$
One easily checks that, for indeterminates $a_1, \ldots, a_m$, 
$$
\sum_{V\subset \llbracket 1,  m\rrbracket, |V|=k}h_{j-k+1}(a_V)=
 \sum_{\lambda \vdash j-k+1} {m-\ell(\lambda)
  \choose k -\ell(\lambda)} m_\lambda(a_1, \ldots, a_m),
$$
so that the above expression of $\Phi^>_{k-1}(u)$ coincides
with~\eqref{phi-rec}. 
\end{proof}

\subsection{The case $y=1$}
As explained in Section~\ref{sec:principle},
Theorem~\ref{thm:main} will be proved if we establish
$\tG(u,1)=G_1(u)$, where  
$$ 
G_1(u)= (1+\bu)e^{K(u)+L} \left((1+u)e^{-mK(u)} - 1\right).
$$ 
A natural attempt would be to set $y=1$ in the expression of
$\tG(u,y)$ that can be derived from Proposition~\ref{prop:extraction},
as we did when $m=1$ in Section~\ref{sec:m=y=1}. However, we have not been able
to do so, and will proceed differently.

We have proved in Proposition~\ref{prop:combi-lin-m} that the series
$\tG(u,y)$ 
satisfies~\eqref{eq-sym} with 
$\Phi_m(v)=v e^{yV(v)}$.
 In particular, $\tG(u,1)$
satisfies~\eqref{eq-sym} with $\Phi_m(v)=v e^{V(v)}$. By
Proposition~\ref{prop:extraction}, this 
equation, together with the initial condition $\tG(-1,1)=0$,
characterizes $\tG(u,1)$. 
It is clear that  $ G_1(-1)=0$. Hence it suffices  to prove  the
following proposition. 

\begin{Proposition}
  The series $G_1(u)$ satisfies~\eqref{eq-sym}  with $\Phi_m(v)=v e^{V(v)}$. 
\end{Proposition}
\begin{proof}
First observe that 
$$ 
G_1(u) = e^{L} \left(vA(u)^{m-1}-\frac{1}{A(u)}\right).
$$ 
 Using  Lemma~\ref{lemma:Lagrange} with $x_i=A_i$, it follows that
\begin{eqnarray*}
\sum_{i=0}^m \frac{G_1(u_i)}{\prod_{j\neq i} (A_i-A_j)} &=& 
0 + (-1)^{m+1} e^{L} \prod_{i=0}^m \frac{1}{A_i}
\quad\quad\mbox{ (by }\eqref{eq:Lagrange-inv-bis}
\mbox{ and } \eqref{eq:Lagrange-poly}\mbox{)}\nonumber \\
&=&
(-1)^{m+1} e^{L+\sum_i K(u_i)} \prod_{i=0}^m \frac{(1+u_i)}{u_i}.
\end{eqnarray*}
By Lemma~\ref{lem:elem} one has $\prod_{i} (1+u_i) = v$ and $\prod_i u_i =
(-1)^{m+1}$, so it only remains to show that
$L + \sum_{i=0}^m K(u_i) = V(v)$.

Recall that $V(v)$ belongs to $v\mathbb{K}[v][ [ z ] ]$ and
that $K(u)=[u^>] V(v)$. Therefore 
Lemma~\ref{P-reconstruct} gives:
$$ 
V(v) = 0 + \sum_{i=0}^m \big( K(u_i) - K(-1) \big)
.$$
But it follows from~\eqref{eq:defK} that
$$
K(-1) = \sum_{k\geq 1} \frac{p_k}{k}z^k
\sum_{i=1}^{k} {(m+1)k\choose k-i} (-1)^{i}
= 
-\sum_{k\geq 1} \frac{p_k}{k}z^k
\frac{k}{(m+1)k} {(m+1)k\choose k}
= \frac{-L}{m+1},
$$
where we have used the identity 
$$
\sum_{i=1}^a {b \choose a-i} (-1)^i =
- {b-1 \choose a-1} = -\frac{a}{b} {b\choose a},
$$ 
valid for $b\geq a \geq 0$, which is easily proved
by induction on $a$.
Therefore  $V(v)=L+\sum_{i=0}^m K(u_i)$, 
and  the proof is complete.
\end{proof}

We have finally proved that $\tG(u,1)=G_1(u)$. As explained in
Section~\ref{sec:principle}, this implies that $F^{(m)}(x,y)=\tG(u,y)$
after the change of variables~\eqref{t-x-param}. In particular,
$F^{(m)}(x,1)=G_1(u)$, and~\eqref{Fx1} is proved. One then
obtains~\eqref{Fx1-2} in the limit $u\rightarrow 0$, using
$$
[u] K(u) = \sum_{k \geq 1} \frac{p_k}{k} 
{(m+1)k \choose k-1} 
z^k.
$$

\subsection{From series to numbers}

We now  derive from~\eqref{Fx1-2} the expression of the character
given in Theorem~\ref{thm:char}. 
We will extract from $F^{(m)}(t,p;1,1)$ the coefficient of $t^n$.
 We find convenient to rewrite the factor $e^L$ occurring in this series as
 $\tz/s$, where $s^m=t$ and $\tz=se^L$ (so that $\tz^m=z$).

Hence
\begin{eqnarray*}
  [t^n]F(t,p;1,1) 
  &=&[s^{mn+1}]\left(\tz-m\sum_{k\geq1} \frac{p_k}{k} \tz^{km+1} 
{(m+1)k\choose k-1}
\right)  \\
&=&\frac{1}{mn+1} [\tz^{mn}]
\left(1-m\sum_{k\geq1} \frac{p_k}{k}
(km+1)\tz^{km} 
{(m+1)k\choose k-1}
\right)e^{(mn+1)L}
\end{eqnarray*}
by the Lagrange inversion formula. This can be rewritten in terms of
$z=\tz ^m$:
\begin{eqnarray*}
  [t^n]F(t,p;1,1) 
 &=&\frac{1}{mn+1} [z^{n}] \left(1-m\sum_{k\geq1} {p_k}
z^{k} 
{(m+1)k\choose k}
\right)e^{(mn+1)L}.
\end{eqnarray*}
The sum inside the brackets is closely related to the derivative of
$L$ with respect to $z$: 
\begin{eqnarray*}
  [t^n]F(t,p;1,1) 
 &=&\frac{1}{mn+1} [z^{n}]
\left(1-{mz}\frac{\partial L} {\partial z}  \right)  e^{(mn+1)L}\\
&=&\frac{1}{mn+1} [z^{n}]
\left(1-\frac{mz}{mn+1}\frac{\partial }{\partial z}  \right)  e^{(mn+1)L}\\
&=&\frac{1}{mn+1} \left(1-\frac{mn}{mn+1}\right)[z^{n}]
e^{(mn+1)L}\\
&=&\frac{1}{(mn+1)^2} [z^{n}] \prod_{k\ge 1}
\exp\left((mn+1)\frac{p_k}{k}z^k
{(m+1)k\choose k}
\right)\\
&=&\frac{1}{(mn+1)^2} 
\sum_{\alpha_1+2\alpha_2+\dots =n}(mn+1)^{\sum \alpha_k}
\prod_k \frac{p_k^{\alpha_k}}{k^{\alpha_k}\alpha_k!}
{(m+1)k\choose  k}^{\alpha_k} 
\\
&=&\frac{1}{(mn+1)^2} \sum_{\la=(\la_1, \ldots) \vdash n}(mn+1)^{\ell(\la)}
\, \frac{p_\la}{z_\la}\ \prod_{i\ge 1} 
{(m+1)\la_i\choose   \la_i}.
\end{eqnarray*}
  The final equation is precisely Theorem~\ref{thm:char}.

\subsection{The complete series $F(t,p;x,y)$}
\label{sec:thm-main-y}
We finally give an explicit expression of the complete series
$F(x,y)\equiv F^{(m)}(t,p;x,y)$. 
Recall that $F(x,y)=\tG(u,y)$ after the change of
variables~\eqref{t-x-param}, and that  the series
$\tG(u,y)$ 
satisfies~\eqref{eq-sym} with $\Phi_m(v)=v e^{yV(v)}$
(Proposition~\ref{prop:combi-lin-m}).  Hence
Proposition~\ref{prop:extraction} gives an explicit, although complicated,
expression of the complete series $F(t,p;x,y)$.

\begin{Theorem}\label{thm:trivariate}
Let  $F^{(m)}(t,p;x,y)\equiv F(t,p;x,y)$ be the refined Frobenius series
of the $m$-Tamari representation, defined by~\eqref{F-def}. Let $z$ and
$u$ be two indeterminates, and write 
$$
t=z e^{-mL}
\quad \hbox{and } \quad x=({1+u})e^{-mK(u)}.
$$
where $L$ and $K(u)$ are defined by {\rm{(\ref{eq:defL}--\ref{eq:defK})}}.
Then $F(t,p;x,y)$ becomes a series in $z$ with polynomial coefficients
in $u$, $y$ and the $p_i$, and this series can be computed by an iterative
extraction of positive parts. More precisely,
\beq\label{F-complete}
F(t,p;x,y)= \sum_{k=0}^m \Phi_k(v) A(u)^k,
\eeq
where $v=u^{-m}(1+u)^{m+1}$, $A(u)$ is defined by~\eqref{A-def}, and 
   $\Phi_k(v)\equiv \Phi_k(z;v)$ is a series in $z$ with polynomial
coefficients in $v$,
$y$ and the $p_i$'s. 
This series can be computed by a
descending induction on $k$ as follows. First,
$\Phi_m(v)=ve^{yV(v)}$ where $V(v)$ is defined by
\eqref{eq:defV}. Then for $1\le k\le m$,
$$
\Phi_{k-1}(v)= \sum_{i=0}^m \left(  \Phi_{k-1}^>(u_i)-  \Phi_{k-1}^>(-1)\right)
$$
where 
\begin{eqnarray}
\Phi_{k-1}^>(u)&=&[u^>]\Phi_{k-1}(v)\nonumber\\
&=&
-\frac 1{{m\choose k} }[u^{>}] \left(\sum_{j=k}^m \Phi_j(v) \sum_{\lambda \vdash j-k+1} {m-\ell(\lambda)
  \choose k -\ell(\lambda)} m_\lambda(A(u_1), \ldots, A(u_m))\right),
\label{phik}
\end{eqnarray}
and $u_0=u, u_1, \ldots, u_m$ are the $m+1$ roots of the equation
$(1+u_i)^{m+1}=u_i^m v$. 
\end{Theorem}

We can rewrite~\eqref{F-complete} in a slightly different form, which 
gives directly~\eqref{F-param-y1} when $m=1$. 
This rewriting combines~\eqref{F-complete}
with the expression of $[u^>] \Phi_{0}(v)$ derived from~\eqref{phik}. 
The case $k=1$ of~\eqref{phik} reads
\beq\label{phi0}
[u^>] \Phi_{0}(v)= -\frac 1{m }[u^{>}] \left(\sum_{j=1}^m \Phi_j(v)
\sum_{i=1}^m A(u_i)^j\right).
\eeq
Recall that  $F(t,p;x,y)=\tG(z,p;u,y)$ has polynomial coefficients in
$u$,
and that $x=1$ when $u=0$.
Hence, returning to~\eqref{F-complete}:
\begin{eqnarray}
  F(t,p;x,y)&=&F(t,p;1,y)+ [u^>] \left( \sum_{k=0}^m
    \Phi_k(v)A(u)^k\right) \nonumber
\\
&=& F(t,p;1,y)+ [u^>] \left( \sum_{k=1}^m \Phi_k(v)\left( A(u)^k -\frac 1{m }
\sum_{i=1}^m A(u_i)^k\right)\right) \hskip 10mm
(\mbox{by } \eqref{phi0}) \nonumber
\\
&=&
({1+u}) [u^{\ge }]\left(\sum_{k=1}^m \frac{\Phi_k(v)}{1+u} \left(
  A(u)^k-\frac 1{m }\sum_{i=1}^m A(u_i)^k\right)\right) \label{F-alt}
\end{eqnarray}
by~\eqref{1+u}, and given that $F(t,p;x,y)=0$ when $u=-1$. 
The proof that $\Phi_k(v) \sum_{i=1}^m A(u_i)^k$ has coefficients in
$(1+u)\GK[u, \bu]$ (which is needed to apply~\eqref{1+u}) is similar to the proof that~\eqref{second-sum:a}
has coefficients in $v\GK[v]$. 

\medskip

\noindent {\bf Examples.} 
We now specialize~\eqref{F-alt} to $m=1$ and $m=2$.
When $m=1$,~\eqref{F-alt} coincides
with~\eqref{F-param-y1} (recall that  $\Phi_m= ve^{yV(v)}$).
When $m=2$, we obtain the following expression for $F^{(2)}$.
\begin{Corollary}
\label{coro:m=2}
Let $V(v), L$ and $ K(u) $ be 
the series given by~{\rm{(\ref{eq:defV}--\ref{eq:defK})}}, with $m=2$.
Perform the change of
variables~\eqref{t-x-param}, still with $m=2$.
   Then the weighted Frobenius series of the $2$-Tamari representation
  satisfies
 \begin{multline*}
     \frac{F^{(2)}(t,p;x,y)}{1+u}=
\\
[u^{\ge }]\left(\frac{\Phi_1(v)}{1+u} \left(
  A(u)-\frac {A(u_1)}{2 }-\frac {A(u_2)}{2 }\right)
+(1+\bu)^2e^{yV(v)} \left(
  A(u)^2-\frac {A(u_1)^2}{2 }-\frac {A(u_2)^2}{2 }
\right)\right),
  \end{multline*}
where 
$$
u_{1,2}=\frac{1+3u\pm (1+u)\sqrt{1+4u}}{2u^2},\quad \quad 
A(u)=\frac{u}{1+u}e^{-K(u)},
$$
and
$$
\Phi_1(v)=   \Phi_{1}^>(u)+\Phi_{1}^>(u_1)+\Phi_{1}^>(u_2)- 3\Phi_{1}^>(-1),
$$
with
$$
\Phi_1^>(u)=-[u^>]\left( (1+u)^3\bu^2e^{yV(v)}\left(A(u_1)+A(u_2)\right)\right).
$$
\end{Corollary}
This expression has been checked with {\sc Maple}, after computing
the first coefficients of $F^{(2)}(t,p;x,y)$ from the functional
equation~\eqref{eq:Fb}.

\section{Final comments}
\label{sec:comments}
\subsection{A constructive proof?}
\label{sec:final-guess}
 Our proof would not
have been possible without a preliminary task consisting in  \emm guessing,  the
expression~\eqref{Fx1} of $F(t,p;x,1)$.  This turned out to be difficult, in particular
because the standard guessing tools, like the {\sc Maple} package {\sc Gfun}, can only guess D-finite
\gfs, while the \gf\ of the numbers~\eqref{caractere}, or
even~\eqref{dimension}, is not D-finite. 
 The expression of $F(t,p;x,1)$ actually \emm becomes,
D-finite in $z$ (at least when only finitely many $p_i$'s are non-zero)
after the change of variables~\eqref{t-x-param}.
The correct parametrization of the
variable $t$ by $z$ was not hard to obtain using the (former)
conjecture~\eqref{caractere} and the Lagrange inversion formula, but we had no indication on the correct
parametrization of $x$. Our discovery of it only came after a long
study of special cases (for instance $m=1$ and $p_i={\mathbbm
  1}_{i=1}$), 
and an analogy with the enumeration of unlabelled Tamari
intervals~\cite{bousquet-fusy-preville}. 
Obviously, a constructive proof of our result would be most welcome,
not to mention  a bijective one.


\subsection{The action of $\Sn_n$ on prime $m$-Tamari intervals}
Other remarkable formulas, as simple as~\eqref{caractere}
and~\eqref{dimension}, can be derived from our expression~\eqref{Fx1} of
the series $F^{(m)}(t,p;x,1)$. Let us for instance focus on the action
of $\Sn_n$ on \emm prime, intervals, that is, intervals $[P,Q]$ such that $P$
has only two contacts with the line $\{x=my\}$. The character $\tilde \chi_m$ of this
representation  is obtained by extracting the coefficient of $x^2$
from $F^{(m)}(t,p;x,1)$, and the Lagrange inversion formula gives,
for a partition $\lambda$ of length $ \ell$:
$$
\tilde \chi_m(\lambda)= \left((m+1)n-1\right)^{\ell-2} \prod_{1 \leq i \leq \ell}
{\binom{(m+1) \lambda_i-1}{\lambda_i}}. 
$$
In particular, the number of  prime labelled $m$-Tamari intervals of size $n$ is
$$ 
((m+1)n-1)^{n-2}m^n.
$$
For \emm
unlabelled,  intervals, it follows from~\cite[Coro.~11]{bousquet-fusy-preville}
that the corresponding numbers are
$$
\frac m{n((m+1)n-1)}{{(m+1)^2n-m-1} \choose {n-1}}.
$$

\subsection{The number of unlabelled $m$-Tamari intervals}
Recall from Lemma~\ref{lem:ordinary} that the series $F^{(m)}(t,p;x,y)$ can also
be understood as the \gf\ of (weighted) unlabelled $m$-Tamari
intervals. In particular, when $p={\bf 1}=(1,1,\ldots)$ and $y=1$, we
have
$$ 
h_k=\sum_{\lambda \vdash k} \frac{1}{z_\lambda}=1,
$$
(because $k!/z_\lambda$ counts permutations of cycle type $\lambda$),
so that
 $$
F^{(m)}(t,{\bf 1};x,1)
=
\sum_{I=[P,Q]  \ \hbox{\small{unlabelled}} }t^{|I|} x^{c(P)}.
$$
By specializing Theorem~\ref{thm:main} to the case $y=1, p={\bf 1}$, we
recover  the following result, already proved
in~\cite{bousquet-fusy-preville}. The result
of~\cite{bousquet-fusy-preville} also keeps track of the size of the
first ascent, but we have not been able to recover it in this generality.
\begin{Proposition}[\cite{bousquet-fusy-preville}]
Let $z'$ and $u'$ be two indeterminates, and write
\begin{eqnarray}\label{eq:paramNonEtiq}
  t=z'(1-z')^{m^2+2m} \ \ \ \ \mbox{ and} \ \ \ \ x=\frac{1+u'}{(1+z'u')^{m+1}}.
\end{eqnarray}
Then the ordinary generating function of unlabelled $m$-Tamari intervals,
counted by the size and the number of contacts, becomes a series in
$z'$ with polynomial coefficients in $u'$, and admits the following
closed form expression: 
\begin{eqnarray}\label{eq:serieNonEtiq} 
 F^{(m)}(t,{\bf 1};x,1)=
\frac{(1+u')(1+z'u')}{u'(1-z')^{m+2}}\left(\frac{1+u'}{(1+z'u')^{m+1}}-1\right).
\end{eqnarray}
\end{Proposition}
\noindent As shown in~\cite{bousquet-fusy-preville}, this implies that the
number of unlabelled $m$-Tamari intervals of size $n$  is
$$
\frac {m+1}{n(mn+1)} {(m+1)^2 n+m\choose n-1}.
$$
\begin{proof}
We need to relate the parametrizations~\eqref{t-x-param} and~\eqref{eq:paramNonEtiq}, and
 then the expressions~\eqref{Fx1} and~\eqref{eq:serieNonEtiq}. Let
 $M\equiv M(z)$ be the unique formal power series in $z$ satisfying
 \begin{eqnarray}\label{eq:Mballot}
     M=1+zM^{m+1}.
   \end{eqnarray}
  We claim that, when $p={\bf 1}$,
\begin{eqnarray}\label{eq:Lquandp=1}
e^{L} = M^{m+1} \quad \hbox{and} \quad   e^{K(u)} =
\frac{1}{1-zuM^{m+1}}=\frac{1}{1-u(M-1)}. 
\end{eqnarray}
This establishes the equivalence between the parametrizations~\eqref{t-x-param}
and~\eqref{eq:paramNonEtiq}, with 
$$
M=\frac 1{1-z'} \quad \hbox{and} \quad u= \frac{u'(1-z')}{1+u'z'}.
$$
 The equivalence between the two expressions of $F^{(m)}$,
 namely~\eqref{Fx1} and~\eqref{eq:serieNonEtiq}, also follows.

\smallskip
We will prove~\eqref{eq:Lquandp=1} using combinatorial interpretations of the
series $K, L, M$ in terms of lattice paths on the square grid, starting at
$(0,0)$ and formed of north and east steps. First, note that $M$ counts
$m$-ballot paths (defined in the introduction) by the size. Also, 
$$
B(z):=z  \frac{d}{dz}L(z,{\bf 1})= \sum_{k\ge 1} {{(m+1)k}\choose k}
z^k
$$ 
counts, by the number of north steps,  non-empty paths ending on the line
$\{x=my\}$ (often called \emm bridges,,
hence the notation $B$). We have
$M=1/(1-P)$, where $P$ counts \emm prime, ballot paths (those that
only have two contacts). 
 By a variant of the cycle   lemma~\cite[Section~4.1]{Ba-Fl}, there exists a size
 preserving bijection between  non-empty bridges and pairs  formed of a prime
 excursion with a marked step, and an 
   excursion. 
Since a bridge having $n$ north steps has $(m+1)n$ steps
   in total, this gives: 
   \begin{eqnarray}\label{eq:cyclelemma}
   z  \frac{d}{dz}L(z,{\bf 1})=  B(z)=\frac{z (m+1) P'(z)}{1-P(z)} =
 z \frac{d}{dz}\left(\ln M(z)^{m+1}\right).
   \end{eqnarray}
 Integrating over $z$ and then exponentiating gives the first part of~\eqref{eq:Lquandp=1}.

Let us now consider the series $K(z,{\bf 1};u)$. We will interpret it
in terms of paths of length $k(m+1)$ for some $k$ (to generalize the
terminology used for ballot paths, we say that such paths have
\emm size, $k$). 
The \emm depth, of path ending
at $(x,y)$ is  $x-my$. Observe that
$$
z\frac{d}{dz} K(z,{\bf 1};u)=
\sum_{k\geq 1} z^k
\sum_{i=1}^{k} {(m+1)k\choose k-i} u^{i},
$$ 
counts paths of length multiple of $(m+1)$ having a positive depth 
($z$ accounts for 
the size, divided by $(m+1)$, and $u$ for the depth, also divided by $(m+1)$).
Let $w$ be such a path, and look at the shortest prefixes of $w$ of
depth 1, then depth 2, and so on up to depth
$(m+1)i$. This factors $w$ into a sequence
$(M_1,e,M_2,e,\ldots,M_{(m+1)i},e,B)$, where the $M_i$ are ballot paths,
$e$ stands for an east step and $B$ is a bridge. Accordingly,
$$
z\frac{d}{dz}K(z,{\bf 1};u) = (1+B(z)) \left(\frac{1}{1-zuM(z)^{m+1}}-1 \right) =
z\frac{d}{dz}\left( \ln \frac{1}{1-zuM(z)^{m+1}}\right),
$$
by \eqref{eq:cyclelemma}.  Integrating and exponentiating gives the second
part of~\eqref{eq:Lquandp=1}.
\end{proof}

\subsection{A $q$-analogue of the functional equation} 
As described in the introduction, the numbers~\eqref{dimension} are
conjectured to give the dimension  of certain polynomial rings
generalizing $\DR_{3  ,n}$. These rings are tri-graded (with respect to
the sets of variables 
$\{x_i\}$, $\{y_i\}$ and $\{z_i\}$), 
 and it is conjectured~\cite{bergeron-preville}
that the dimension of the homogeneous component in the $x_i$'s of degree $k$ 
is the number of labelled intervals $[P,Q]$ in $\cT_{n}^{(m)}$ such that the
longest chain from $P$ to $Q$, in the Tamari order, has length
$k$. One can recycle the recursive description of intervals described
in Section~\ref{sec:eq} to generalize the functional equation of
Proposition~\ref{prop:eq} (taken when $p_i={\mathbbm
  1}_{i=1}$), 
by taking into account (with a new variable $q$) this
distance. Eq.~\eqref{eq:Fb} becomes
$$
\frac{\partial F}{\partial y} (x,y)=
 tx(F(x,1)\Delta)^{(m)}(F(x,y)) ,
$$
where now
$$
\Delta S(x)=\frac{S(qx)-S(1)}{qx-1}.
$$
Here $F(1,1)$ counts labelled $m$-Tamari intervals by the size and the above
defined distance. But we have not been able to conjecture any simple $q$-analogue
of~\eqref{dimension}.

\spacebreak

\bigskip
\noindent{\bf Acknowledgements.} We are grateful to François Bergeron
for advertising in his  lectures various  beautiful conjectures  related to
 Tamari intervals.
   We also thank \'Eric Fusy and Gilles Schaeffer   for
interesting discussions on this topic, and  thank \'Eric once more for
allowing us to reproduce some figures
of~\cite{bousquet-fusy-preville}. 
%

\bibliographystyle{plain}
\bibliography{tamar}

\spacebreak

\end{document}